\documentclass{article}
\usepackage{amssymb}
\usepackage{amsmath}
\usepackage{amsthm}
\usepackage{amscd}
\usepackage{graphicx}
\usepackage{geometry}
\usepackage[breaklinks,pdfstartview=FitH,colorlinks,linktocpage,linkcolor=blue,citecolor=blue,urlcolor=blue]{hyperref}

\usepackage[cmtip,all]{xy}
\usepackage{upref}

\newtheorem{theorem}{Theorem}[section]
\newtheorem{lemma}[theorem]{Lemma}

\newtheorem{corollary}[theorem]{Corollary}

\theoremstyle{definition}

\newtheorem*{definition*}{Definition}
\newtheorem{example}[theorem]{Example}

\theoremstyle{remark}
\newtheorem{remark}[theorem]{Remark}
\newtheorem*{remark*}{Remark}

\numberwithin{equation}{section}

\DeclareMathOperator{\rank}{rk}

\DeclareMathOperator{\interior}{Int}

\DeclareMathOperator{\image}{Im}
\DeclareMathOperator{\kernel}{Ker}

\begin{document}

\large

\title{When is the set of embeddings finite up to isotopy?}

\author{Mikhail Skopenkov}

\date{}
\maketitle

\begin{abstract}
Given a manifold $N$ and a number $m$, we study the following question: \emph{is the set of isotopy classes of embeddings $N \to S^m$ finite}?
In case when the manifold $N$ is a sphere the answer was given by A. Haefliger in 1966.
In case when the manifold $N$ is a disjoint union of spheres the answer was given by
D. Crowley, S. Ferry and the author in 2011.
We consider the next natural case when $N$ is a product of two spheres. In the following theorem, $FCS (i,j)\subset \mathbb{Z}^2$ is a specific set depending only on the parity of $i$ and~$j$ which is defined in the paper.

\textbf{Theorem.} Assume that $m>2p+q+2$ and $m<p+3q/2+2$. Then the set of $C^1$-isotopy classes of $C^1$-smooth embeddings $S^p \times S^q \to S^m$ is infinite if and
only if either $q+1$ or $p+q+1$ is divisible by $4$, or there exists
a point $(x,y)$ in the set $FCS (m-p-q,m-q)$ such that $(m-p-q-2)x+(m-q-2)y=m-3$.


Our approach is based on a group structure on the set of embeddings and a new exact sequence, which in some sense reduces the classification of embeddings $S^p \times S^q \to S^m$ to the classification of embeddings $S^{p+q} \sqcup S^q \to S^m$ and $D^p \times S^q \to S^m$.
The latter classification problems are reduced to homotopy ones, which are solved rationally.


\end{abstract}

\maketitle

\footnotetext[0]{
This is an improved version of the paper published in Intern. J. Math 26:7 (2015), 28 pp.\\
The article was prepared within the framework of the Academic Fund Program at the National Research University Higher School of Economics (HSE) in 2015-2016 (grant No 15-01-0092)  and supported within the framework of a subsidy granted to the HSE by the Government of the Russian Federation for the implementation of the Global Competitiveness Program.
During the work on this paper the author received support
also from ``Dynasty'' foundation and from the Simons--IUM fellowship. 
}

\section{Introduction}\label{sect1}

This paper is on the classification of embeddings of
higher-dimensional manifolds, see \cite{Sko07L} for a recent survey.
This generalizes the subject of classical knot theory.
In general one can hope only to reduce the isotopy classification problem to problems of homotopy theory \cite{Hae66A, Hae66C, Kos90, Le65}. Sometimes the latter can be solved but finding explicit classification is hard.

Given a manifold $N$ and a number $m$, we study the following simpler question: \emph{is the set of isotopy classes of embeddings $N \to S^m$ finite}?
This question is motivated by analogy to rational homotopy theory founded by J.P.~Serre, D.~Sullivan and D.~Quillen \cite{GM81} and rational classification of link maps by U.~Koschorke \cite{Kos90, HaKa98}.
We give answers for simplest manifolds $N$: spheres, disjoint unions of spheres (known before) and products of two spheres (new).
Our main result (Theorem~\ref{th3} below) is an exact sequence, which in some sense reduces the classification of embeddings $S^p \times S^q \to S^m$ to the classification of embeddings $S^{p+q} \sqcup S^q \to S^m$ and $D^p \times S^q \to S^m$. This provides much information about the set of isotopy classes of embeddings $S^p \times S^q \to S^m$ including a finiteness criterion (Theorem~\ref{th2} below).
Some results for general manifolds $N$ are available only in so-called metastable dimension 
\cite{Sko07L, Kl05}.
Throughout the paper we work in $C^1$-smooth category.

This paper concludes the series of papers~\cite{CRS07,CRS08,CFS11}. 
It is independent of previous ones in the sense that it uses statements from~\cite{CFS11} but neither definitions nor methodology from any of them.

\subsection*{Knots and links} For {\it knots} $S^q\to S^m$ in codimension at least $3$
(i.e., $m>q+2$) the answer to the posed question is given by A.~Haefliger:

\begin{theorem} \label{th1} \textup{\cite[Corollary~6.7]{Hae66A}} Assume
that $m> q+2$. Then the set of smooth isotopy classes of
smooth embeddings $S^q\to S^m$ is infinite if and only if $m < {3q}/{2}+2$ and $q+1$ is divisible by $4$.
\end{theorem}

The classification of (\emph{partially}) \emph{framed knots} $D^p\times S^q\to S^m$ is closely related to the classification of knots. 

\begin{theorem}\label{th1f}
\textup{\cite[Corollary~1.14]{CFS11}}
Assume that $m>q+2$, $1\le p\le m-q$. Then the set of smooth isotopy classes of smooth embeddings $D^p\times S^q\to S^m$ is infinite if and only if one of the following conditions holds:
\begin{itemize}
\item $4\,|\,q+1$ and $m<p+3q/2+1$;
\item $2\,|\,q+1$ and $m=2q+1$;
\item $2\,|\,q$ and $m=p+2q$.
\end{itemize}
\end{theorem}


The classification of {\it links} $S^p\sqcup S^q\to S^m$ is the next natural problem after the classification of knots. For $m\ge {2}(p+q)/3+2$ there is an explicit description of the isotopy
classes of links $S^p\sqcup S^q\to S^m$ ``modulo'' knots $S^p\to S^m$
and $S^q\to S^m$ in terms of homotopy groups of spheres and Stiefel
manifolds; see \cite[Theorem 10.7]{Hae66C}, \cite[Theorem~1.1]{Sko08P}, \cite{Ne82}.
In codimension at least $3$ there is an
exact sequence involving the set of isotopy classes of links and certain
homotopy groups \cite[Theorem~1.3]{Hae66C}.
This sequence allows to obtain the following finiteness criterion by D.~Crowley, S.~Ferry and the author.
The criterion involves certain \emph{finiteness-checking} sets $FCS (i,j)\subset \mathbb{Z}^2$ which depend only on the parity of $i$, $j$ and which are defined in Table~\ref{formal} below.
A part of each set is drawn in the table; the rest of the set is obtained by obvious periodicity.

\begin{theorem}\label{th1b} \textup{\cite[Theorem~1.5]{CFS11}} Assume that $p,q< m-2$. Then the set of smooth isotopy classes of smooth embeddings $S^p\sqcup S^q\to S^m$, whose components are unknotted, is infinite if and only if
there exists a point $(x,y)\in FCS (m-p,m-q)$ such that $(m-p-2)x+(m-q-2)y=m-3$.
\end{theorem}



\begin{table}[th]
\caption{Definition of the finiteness-checking set $FCS (i,j)$ \cite[Table~1]{CFS11}}
\label{formal}
\begin{center}
\begin{tabular}{|p{4cm}|p{4cm}|p{4cm}|}
\hline
\multicolumn{3}{|p{12cm}|}{$FCS (i,j)$ is the set of pairs $(x,y)\in \mathbb{Z}\times\mathbb{Z}$ such that $x,y>0$ and at least one of the following conditions holds~---}\\[2pt]
\hline
\multicolumn{1}{|c|}{for $i,j$ even:} & \multicolumn{1}{|c|}{for $i$ odd, $j$ even:} & \multicolumn{1}{|c|}{for $i,j$ odd:}\\[2pt]
\hline
\begin{itemize}
\item $x=1$ and $y=1$;
\item $x=2$ and $2\,|\,y$;
\item $x=3$ and $y=3$;
\item $x=3$ and $y\ge5$;
\item $x\ge4$ and $y\ge4$;
\item $2\,|\,x$ and $y=2$;
\item $x\ge5$ and $y=3$.
\end{itemize}
&
\begin{itemize}
\item $x=1$ and $y=1$;
\item $x=2$ and $2\,|\,y+1$;
\item $x=3$ and $y\ge2$;
\item $x\ge4$ and $y\ge4$;
\item $4\,|\,x$ and $y=2$;
\item $4\,|\,x+1$ and $y=2$;
\item $x\ge5$ and $y=3$.
\end{itemize}
&
\begin{itemize}
\item $x=1$ and $y=1$;
\item $x=2$ and $4\,|\,y+2$;
\item $x=2$ and $4\,|\,y+3$;
\item $x\ge 3$ and $y\ge3$;
\item $4\,|\,x+2$ and $y=2$;
\item $4\,|\,x+3$ and $y=2$.
\end{itemize}\\
\hline
& & \\[2pt]
\multicolumn{1}{|c|}{\quad \includegraphics[width=3.5cm]{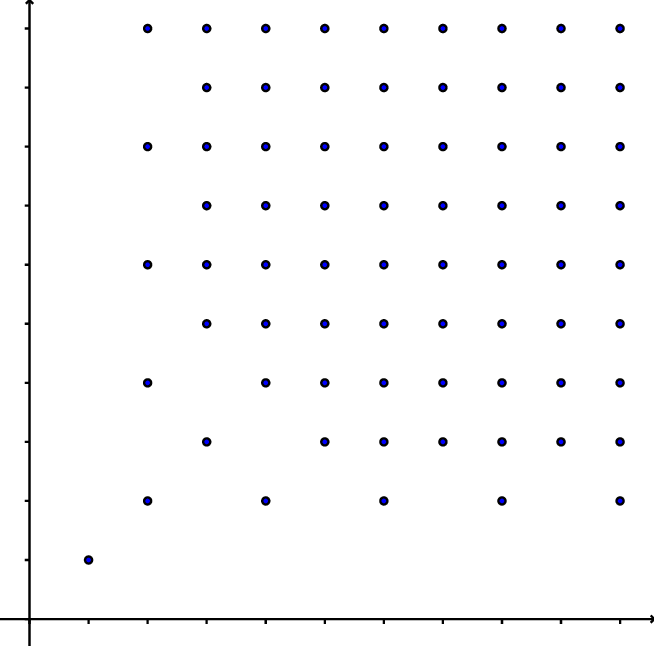} \quad} & \multicolumn{1}{|c|}{\quad \includegraphics[width=3.5cm]{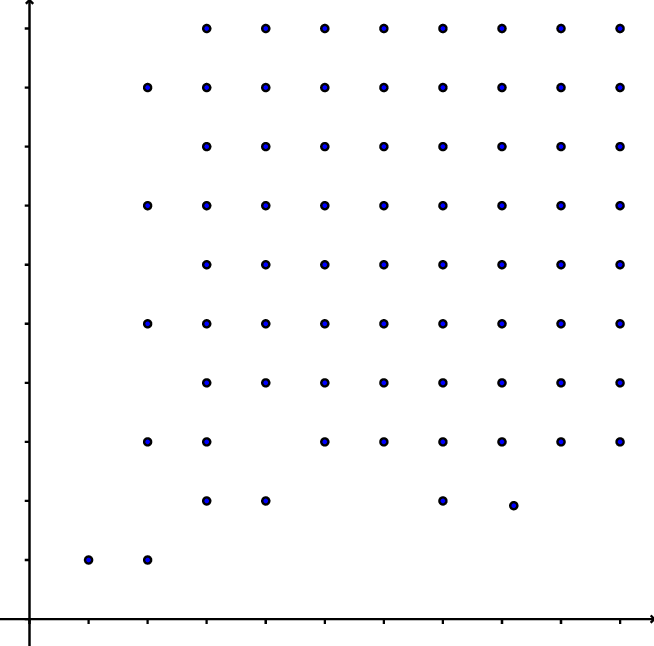} \quad} & \multicolumn{1}{|c|}{\quad \includegraphics[width=3.5cm]{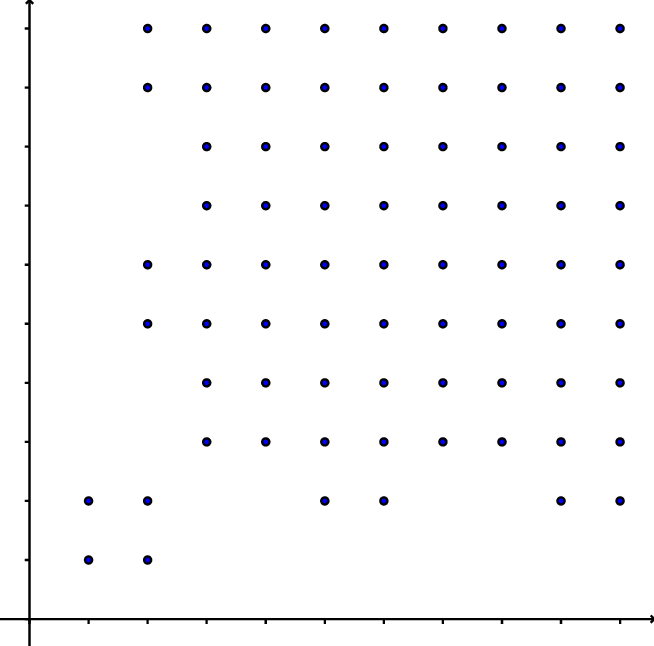} \quad} \\[5pt]
\hline
\multicolumn{3}{|p{12cm}|}{For $i$ even, $j$ odd the set $FCS (i,j)$ is obtained from $FCS (j,i)$ by the reflection with respect to the line $x=y$.}\\[2pt]
\hline
\end{tabular}
\end{center}
\end{table}

\subsection*{Knotted tori}

A natural next step (after link theory and the classification of embeddings of highly-connected manifolds) towards classification of embeddings of arbitrary manifolds is the classification of {\it knotted tori}, i.e., embeddings $S^p\times S^q\to S^m$.
The classification of knotted tori gives some insight or even precise information concerning arbitrary manifolds \cite{Sko07F}; see also Theorem~\ref{thconsum} below.
Many interesting examples of embeddings are knotted tori \cite{Hud63, MiRe71, Sko02, Sko10}.


There was known an explicit description of the set of isotopy classes of knotted tori ``modulo'' knots $S^{p+q}\to S^m$
in the {\it metastable} dimension $m\ge p+{3q}/{2}+2$, $p\le q$, in terms of homotopy groups of Stiefel manifolds
\cite{
Sko02, Sko08}. 
If $N$ is a closed $(p-1)$-connected $(p+q)$-manifold 
then until recent results \cite{CrSk08, Sko08Z, Sko08} no complete readily calculable descriptions of isotopy classes below the metastable dimension was known, in spite of the existence of interesting
approaches of Browder--Wall and Goodwillie--Weiss \cite{Wa65, GW99, CRS04}.

The main ``practical'' result of the paper is an
explicit criterion for the finiteness of the set of knotted tori
up to isotopy below the metastable dimension:

\begin{theorem} \label{th2}
Assume that $m > 2p+q+2$ and $m <
p+{3q}/{2}+2$. Then the set of isotopy classes of smooth
embeddings $S^p\times S^q\to S^m$ is infinite if and
only if at least one of the following conditions holds:
\begin{itemize}
\item $q+1$ or $p+q+1$ is divisible by $4$,
\item there exists a point $(x,y)\in FCS (m-p-q,m-q)$ such that $(m-p-q-2)x+(m-q-2)y=m-3$.
\end{itemize}
\end{theorem}

\begin{example} \label{ex1} \cite[Example~1]{CRS08} The set of knotted tori $S^1\times S^5\to
S^{10}$ is finite up to isotopy.
\end{example}






In Theorem~\ref{th2} the inequality $m < p+{3q}/{2}+2$ is assumed by aesthetic reasons --- to reduce the number of cases and thus to simplify the statement and the proof. The 
classification of knotted tori for $m \ge p+{3q}/{2}+2$
is easier and is given by \cite[Corollary 1.5]{Sko02}, \cite[Theorem 1.2]{Sko08}, \cite[Lemma~1.12]{CFS11}.

A particular case $m>p+{4q}/{3}+2$ of Theorem~\ref{th2} was proved in \cite{CRS07, CRS08} by a different method.
To compare roughly the strength of all the mentioned results one could put $p=1$. Then
known results provide a classification for $m>3q/2+3$,
while Theorem~\ref{th2} provides a finiteness criterion for $m > q+4$.

\subsection*{Relationship between framed knots, links and knotted tori}

The main result of the paper is an exact sequence (Theorem~\ref{th3} below), which in some sense reduces the classification of knotted tori to the classification of links and framed knots.

Let us introduce some notation and conventions.
For a smooth manifold $N$ denote by $E^m (N)$ the set of smooth isotopy classes of smooth embeddings $N\to S^m$. The letter ``$E$'' in the notation comes from the word ``embedding''.
For $m> p+q+2$  the sets $E^m (S^q)$, $E^m (D^p\times S^q)$, and $E^m(S^{p+q}\sqcup S^q)$ are finitely generated Abelian groups with respect
to ``connected sum'', ``framed connected sum'', and ``componentwise connected sum'' operation, respectively \cite{Hae66A, Hae66C}.
Denote by $E_0^m(S^{p+q}\sqcup S^q)$ the subgroup of $E^m(S^{p+q}\sqcup S^q)$ formed by all the embeddings $S^{p+q}\sqcup S^q\to S^m$ whose second component (i.e., restriction to the sphere $S^q$) is unknotted.
For $m > 2p+q+2$ the ``parametric
connected sum'' operation gives a natural Abelian group structure on the set $E^{m}(S^p\times S^q)$; see Figure~\ref{fig5}, Section~\ref{sectprem}, and paper~\cite[\S2.1]{Sko15} for details.


\begin{figure}[htb]
\includegraphics{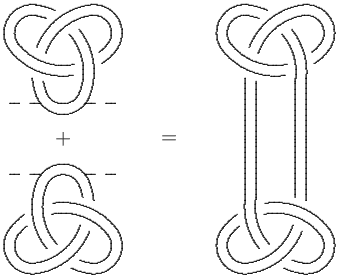}
\caption{Informal illustration of $S^1$-parametric connected sum of embeddings $S^1\times S^1\to S^3$ \cite[Figure~5]{CRS08}}
\label{fig5}
\end{figure}



\begin{theorem}\label{th3} For each 
$m > 2p+q+2$ 
there is an exact sequence of finitely generated Abelian groups
\begin{multline*}
{\dots}                                                   \xrightarrow{}
{{E}_0^{m}(S^{p+q}\sqcup S^{q})}                          \xrightarrow{\sigma^*}
{{E}^m(S^p\times S^q)}                                    \xrightarrow{i^*}
{{E}^m(D^p\times S^q)}                                    \xrightarrow{\partial^*}\\
                                                          \xrightarrow{}
{{E}_0^{m-1}(S^{p+q-1}\sqcup S^{q-1})}                    \xrightarrow{\sigma^*}
{{E}^{m-1}(S^{p}\times S^{q-1})}                          \xrightarrow{i^*}
{{E}^{m-1}(D^{p}\times S^{q-1})}                          \xrightarrow{}
{\dots}
\end{multline*}
\end{theorem}

For $m\ge 3(p+q)/2+2$
this sequence is isomorphic to the middle horizontal sequence in \cite[Restriction Lemma 5.2]{Sko08}, while for general $m$ our exact sequence can be called a ``desuspension'' of that one; see Remark~\ref{rem-desuspension} below for details.

As a nontrivial corollary, we get the following formula for the rank of the group $E^m(S^p\times S^q)$:

\begin{corollary} \label{essentials}
Assume that $m>2p+q+2$ and $m<p+{3}q/{2}+2$.
Then
$$
\rank E^m(S^{p}\times S^q) =
\rank E^m(S^{p+q}\sqcup S^q)+ \rank \pi_{q}(V_{m-q,p}).
$$
\end{corollary}

Notice that the ranks of the groups in the right-hand side are known \cite[Theorem~1.7 and Lemma~1.12]{CFS11}.

\subsection*{Organization of the paper}
In Section~\ref{sectprem} we introduce some notation and recall some required known results. In Section~\ref{sect2} we prove Theorem~\ref{th3}. In Section~\ref{sect3} we deduce Theorem~\ref{th2} from Theorem~\ref{th3} and give an easy application (Theorem~4.1) of our approach.

The reader who wants to get a nontrivial result in a minimal time may read only subsections ``Group Structure'', ``Definition of $\sigma^*$'', ``Exactness at $E^m(S^p\times S^q)$'', 
and then immediately get the ``only if'' part of Theorem~\ref{th2} from Theorems~\ref{th1}--\ref{th1b}.

Most of the ideas of the paper can be understood from the low-dimensional examples shown in figures. Notice that the proofs may not be literally correct for the shown low dimensions. In the figures instead of the spheres $S^1,S^2,S^3$ we always show their images under 
an appropriate stereographic projection.



\section{Preliminaries} \label{sectprem}


\subsection*{Group structure}

First we define of a group structure
on the set of knotted tori by A. Skopenkov~\cite[\S2.1]{Sko15}.

Let $x_1,x_2,\dots,x_{m+1}$ be the coordinates in space $\mathbb{R}^{m+1}$. For each $q\le m$ identify space $\mathbb{R}^{q+1}$ with the subspace of $\mathbb{R}^{m+1}$ given by the equations $x_{q+2}=x_{q+3}=\dots=x_{m+1}=0$. 
Denote by $r_{k}\colon \mathbb{R}^{m+1}\to \mathbb{R}^{m+1}$
the reflection in the hyperplane given by the equation $x_k=0$.
Denote by $D^q_+$ and $D^q_-$ the half-spheres of the unit sphere $S^q$ 
given by the inequalities $x_{q+1}\ge 0$ and $x_{q+1}\le 0$, respectively.
Then $\partial D^q_+=\partial D^q_-=S^{q-1}$.
Denote by $D^q_{+\epsilon}$ (respectively, $D^{q}_{++}$) the subset of $S^q$ 
given by the inequality $x_{q+1}\ge \epsilon$ (respectively, the inequalities $x_{q+1},x_q\ge 0$).
Denote by $D^q_{1\pm\epsilon}=(1\pm\epsilon)D^q$ the scaling of the unit disc $D^q$.
Let $\alpha_\epsilon\colon S^{q-1}\to \partial D^q_{+\epsilon}$ be the central projection  from the point $(0,\dots,0,(1+\sqrt{1-\epsilon^2})/\epsilon)$.
Fix an embedding $\beta\colon D^{p+q}_-\to\mathrm{Int}(D^p_+\times D^q_+)$ and denote $B=\beta(D^{p+q}_-)$.

If no confusion arises we denote a map and its abbreviation by the same symbol, and also an embedding and its isotopy class by the same symbol.

For $q\le m$ {\it the standard embedding} $s\colon S^q\to {S}^m$ is given by the formula $s(x_1,\dots,x_{q+1})=(0,\dots,0,x_1,\dots,x_{q+1})$, where the number of zero coordinates equals $m-q$. It does \emph{not} coincide with the inclusion $S^q\to S^m$,
$(x_1,\dots,x_{q+1})\mapsto (x_1,\dots,x_{q+1},0,\dots,0)$, coming from the identification of space $\mathbb{R}^{q+1}$ with a subspace of $\mathbb{R}^{m+1}$.

For $p+q\le m$ {\it the standard embeddings} $s\colon \mathbb{R}^p\times S^q\to {S}^{m},D^p\times S^q\to {S}^{m}, S^{p-1}\times S^q,S^{p-1}\times D^{q}_{\pm}\to {D}^{m}_{\pm}$
are given by one formula
$$
s\left((x_1,\dots, x_{p}),(y_1,\dots,y_{q+1})\right)= \frac{(\overbrace{0,\dots,0}^{m-p-q},x_1,\dots, x_{p},y_1,\dots,y_{q+1})}{\sqrt{x_1^2+\dots+x_{p}^2+y_1^2+\dots+y_{q+1}^2}}.
$$
Clearly, $s(D^p\times S^q)$ is the $\frac{\pi}{4}$-neighborhood of the sphere $s(S^q)$ in the sphere $s(S^{p+q})$.



\begin{definition*} (See Figure~\ref{surgery} to the left.)
An embedding $f\colon S^p\times S^q\to S^m$ is {\it standardized}, if 
\begin{itemize}
\item
$f\colon S^p\times D^q_-\to D^m_-$ is the standard embedding;
\item
$f(S^p\times\interior D^q_+)\subset\interior D^m_+$.
\end{itemize}
An isotopy $f_t\colon S^p\times S^q\to S^m$ is {\it standardized}, if for each $t\in I$ the embedding $f_t$ is standardized.
%
\end{definition*}

\begin{lemma}\label{l1} \textup{\cite[Standardization Lemma 2.1]{Sko15}}
Assume that $m>2p+q+2$.
Then

\textup{(a)} any embedding $S^p\times S^q\to S^m$ is isotopic to a standardized embedding; and

\textup{(b)} if two standardized embeddings $S^p\times S^q\to S^m$ are isotopic then there is a standardized isotopy between them.
\end{lemma}


\begin{lemma}\label{l3} \textup{\cite[Group Structure Lemma 2.2]{Sko15}}
Assume that $m> 2p+q+2$. Then an Abelian group structure on the set $E^m(S^p\times S^q)$ is well-defined by the following construction.
\begin{itemize}
\item Let $f,g\colon S^p\times S^q\to S^m$ be two embeddings.
Take standardized embeddings $f',g'$ isotopic to them.
By definition, set $f+g$ to be the isotopy class of the embdedding
$$h(x,y)=
\begin{cases} f'(x,y),& y\in D^q_+;\\ r_mr_{m+1} g'(x,r_q r_{q+1}y),& y\in D^q_-.\end{cases}$$
\item Set $-f$ to be the isotopy class of the embedding $\bar f(x,y)=r_m f(x,r_q y)$.
\item Set $0$ to be the isotopy class of the standard embedding $s\colon S^p\times S^q\to S^m$.
\end{itemize}
\end{lemma}


\begin{lemma}\label{l2} \textup{\cite[Lemma~3.1]{Sko15}} 
Assume that $m>2p+q+2$.
Then
an embedding
$S^p\times S^q\to S^m$ is isotopic to the standard embedding if and only if it
extends to an embedding $S^p\times D^{q+1}\to D^{m+1}$.
\end{lemma}


These lemmas are the only results of~\cite{Sko15} used in the present paper.

The next two subsections give some insight for the proof of Theorem~\ref{th3} although they are not used formally in that proof.

\subsection*{Action of knots}
Let us define an action of the group of knots on the set of embeddings and prove a particular case of Theorem~\ref{th2} (see the paragraph after Lemma~\ref{cl1}).

Define a map
$\sigma^*\colon E^{m}(S^{p+q})\to E^{m}(S^p\times S^q)$ as follows. Represent an element of $E^m(S^{p+q})$ by an embedding $g\colon S^{p+q}\to S^m$ such that the images $g(S^{p+q})$ and $s(S^p\times S^q)$ are separated by a hyperplane. Join these images by an arc whose interior misses the images. Let $\sigma^*(g)\colon S^p\times S^q\to S^m$ be the embedded connected sum $g\#s$ of the knot $g$ and the standard embedding
$s\colon S^p\times S^q\to S^m$ along this arc. For $p=0$ or $q=0$ the manifold $S^p\times S^q$ is disconnected, thus we need to specify that the endpoints of the arc belong to $g(D^{p+q}_+)$ and $s(D^p_+\times D^q_+)$. Clearly, for $m>p+q+2$ the map
$\sigma^*\colon E^{m}(S^{p+q})\to E^{m}(S^p\times S^q)$ is well-defined by this construction; cf. Lemma~\ref{clwelldef} below.

Description of a similar action $E^m(S^n)\times E^m(N)\to E^m(N)$ for a general $n$-manifold $N$ is a hard open problem \cite{CrSk08, Sko08Z}. Fortunately, in our situation enough information can be obtained:

\begin{lemma}\label{cl1} \textup{\cite[Proposition~8]{CRS08}} For $m> 2p+q+2$ the
map $\sigma^*\colon E^{m}(S^{p+q})\to E^{m}(S^p\times S^q)$ is injective.
\end{lemma}

This lemma immediately implies the case ``$p+q+1$ divisible by 4'' of Theorem~\ref{th2}, by Theorem~\ref{th1} above. In fact $E^m(S^{p+q})$ is a direct summand in $E^m(S^p\times S^q)$ \cite[Theorem~1.1]{Sko15}.

Let us define \emph{the standard surgery} used in the proof of Lemma~\ref{cl1}. 
Informally, it is a surgery over the ``torus'' $S^p\times S^q$ along a ``meridian'' $S^p\times y$, $y\in S^q$; see Figure \ref{surgery}.

\begin{figure}[htb]
\begin{tabular}[t]{c}
\\
\includegraphics[width=8cm]{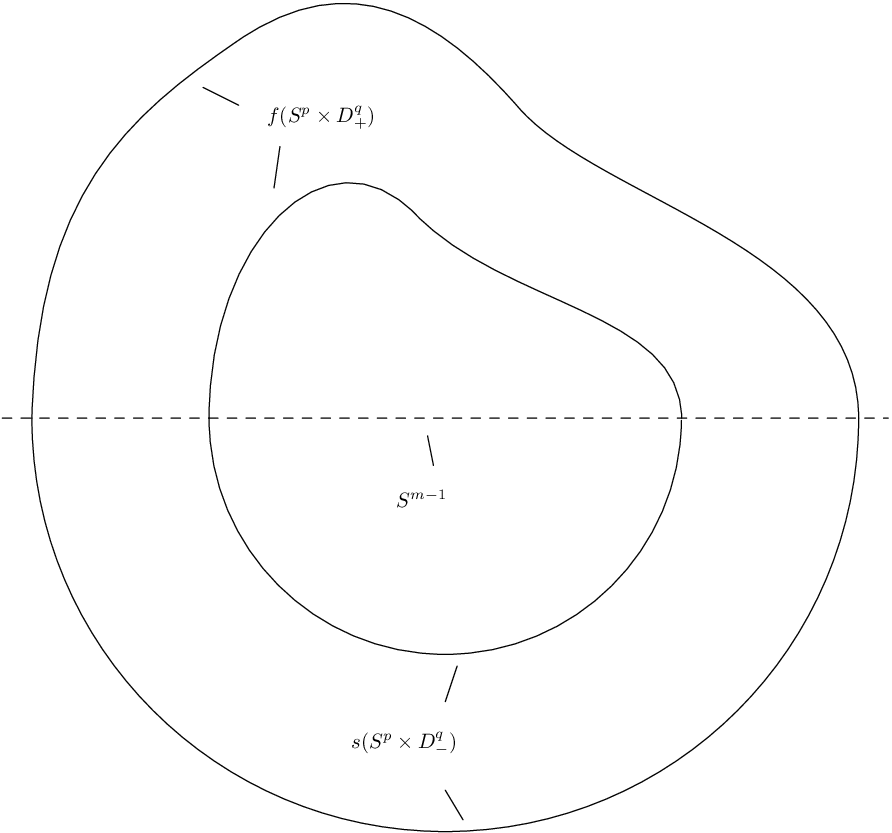}
\end{tabular}
\begin{tabular}[t]{c}
\\[3cm]
$\mapsto$
\end{tabular}
\begin{tabular}[t]{c}
\\
\includegraphics[width=8cm]{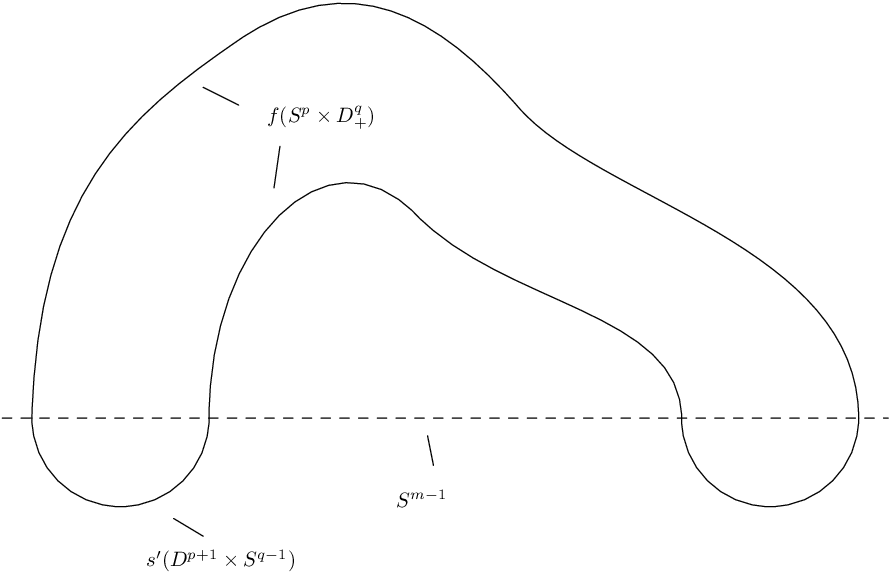}
\end{tabular}
\caption{The standard surgery ($p=0, q=1, m=2$)}\label{surgery}
\end{figure}

To give a formal definition, fix a diffeomorphism $S^{p+q}\cong S^p\times D^q_+\cup D^{p+1}_-\times S^{q-1}$.
Let the embedding $s'\colon D^{p+1}_-\times S^{q-1}\to D^m_-$ be given by the formula
$$s'((x_1,\dots,x_{p+2}),(y_1,\dots,y_q))=
\frac{1}{\sqrt{2}}(\overbrace{0,\dots,0}^{m-p-q-1},x_1,\dots,x_{p+1},y_1,\dots,y_q,x_{p+2}).
$$
Define the result of the {\it standard surgery} over the standard embedding $s\colon S^p\times S^q\to S^m$ to be the $C^1$-smooth unknotted embedding $S^{p+q}\to S^m$ obtained by gluing $s\colon S^{p}\times D^{q}_+\to D^m_+$ and $s'\colon D^{p+1}_-\times S^{q-1}\to D^m_-$ together.
Define the result of the {\it standard surgery} over a standardized embedding $f\colon S^p\times S^q\to S^m$ to be the embedding $g\colon S^{p+q}\to S^m$ obtained by gluing $f\colon S^{p}\times D^{q}_+\to D^m_+$
and $s'\colon D^{p+1}_-\times S^{q-1}\to D^m_-$ together:
$$
g(x)=\begin{cases}
f(x), & x\in S^{p}\times D^{q}_+;\\
s'(x), & x\in D^{p+1}_-\times S^{q-1}.
\end{cases}
$$


\begin{proof}[Proof of Lemma~\ref{cl1}]
It suffices to construct a left inverse $\bar\sigma^*\colon E^{m}(S^p\times S^q)\to E^{m}(S^{p+q})$ of~$\sigma^*$.
Take an element of $E^{m}(S^p\times S^q)$. By Standardization Lemma~\ref{l1}.a it can be realized by a standardized embedding $f\colon S^p\times S^q\to S^m$.
Set $\bar\sigma^*(f)\colon S^{p+q}\to S^m$ to be the result of the standard surgery over $f$.

Let us prove that the isotopy class of $\bar\sigma^*(f)$ is well-defined, i.e., does not depend on the choice of 
$f$ within an isotopy class. Take two isotopic standardized embeddings $f_1,f_2\colon S^{p}\times S^q\to S^m$. By Standardization Lemma~\ref{l1}.b there is a standardized isotopy $f_t$ between $f_1$ and $f_2$. Then $\bar\sigma^*(f_t)$ is an isotopy between $\bar\sigma^*(f_1)$ and $\bar\sigma^*(f_2)$. That is, the isotopy classes of $\bar\sigma^*(f_1)$ and $\bar\sigma^*(f_2)$ are the same.

Let us prove that $\bar\sigma^*\sigma^*=\mathrm{Id}$. Take an element of $E^m(S^{p+q})$.
Represent it by an embedding $g\colon S^{p+q}\to S^m$ such that $g(S^{p+q})$ and $s(S^p\times S^q)$ are separated by a hyperplane.
Then $\bar\sigma^*\sigma^*(g)=\bar\sigma^*(s\# g)=\bar\sigma^*(s)\#g=0\#g=g$.
\end{proof}

\subsection*{Framed knots}

Let us recall an approach to the classification of (partially)
framed knots.

By a {\it $p$-framing} of an embedded manifold we mean a system of $p$ ordered orthogonal normal unit vector fields on the manifold.
Denote by $V_{m-q,p}$ the {\it Stiefel manifold} of $p$-framings of the
origin of $\mathbb{R}^{m-q}$.
Clearly, the group $E^m(D^p\times S^q)$ is isomorphic to the group of $p$-framed embeddings $S^q\to S^m$ up to $p$-framed isotopy.
The group structure on the set $E^m(D^p\times S^q)$ is constructed literally as on $E^m(S^p\times S^q)$ above (with $S^p$ replaced by $D^p$ in Lemmas~\ref{l1}--\ref{l3}) or equivalently as on $(m-q)$-framed embeddings $S^q\to S^m$ up to $(m-q)$-framed isotopy 
in \cite{Hae66A}.

The following result in some sense reduces the classification of framed knots to the classification of knots and  computations of homotopy groups.

\begin{theorem}\label{th4} \textup{\cite[Theorem~9,4)]{CRS08}} For $m>q+2$ 
there is an exact sequence
$$
\dots \to \pi_{q}(V_{m-q,p}) \xrightarrow{} {E}^m(D^p\times S^q) \xrightarrow{} E^{m}(S^q)
\xrightarrow{} \pi_{q-1}(V_{m-q,p})\to {E}^{m-1}(D^p\times S^{q-1})\to\dots
$$
\end{theorem}

The theorem is proved analogously to
its particular case $m=p+q$
\cite[Corollary~5.9]{Hae66A}. We sketch the proof here for convenience of the reader.

\begin{proof}[Sketch of the proof]
{\it Definition of homomorphisms.} The map $i^*\colon {E}^m(D^p\times S^q) \to E^{m}(S^q)$ is the composition ${E}^m(D^p\times S^q) \to {E}^m(0\times S^q) \cong E^{m}(S^q)$
of the restriction-induced map and the map induced by the standard embedding $s\colon S^q\to 0\times S^q$.

The map $Ob\colon E^{m}(S^q)\to \pi_{q-1}(V_{m-q,p})$ is the  obstruction to
the existence of a $p$-framing on an embedding $f\colon S^q\to S^m$
defined as follows. Take a (unique up to homotopy)
$(m-q)$-framing
of the disc $fD^q_+$. Take a (unique up to homotopy) $p$-framing of the disc $fD^q_-$. Thus the sphere $fS^{q-1}$ is equipped both with the $p$-framing and the $(m-q)$-framing. Using the $(m-q)$-framing identify each fiber of the normal bundle to $fD^q_+$ with the space $\mathbb{R}^{m-q}$. To each point $x\in S^{q-1}$ assign the $p$-framing at the point $fx$. This leads to a map $S^{q-1}\to V_{m-q,p}$. By definition $Ob(f)\in \pi_{q-1}(V_{m-q,p})$ is the homotopy class of this map.

The map $\tau\colon \pi_{q}(V_{m-q,p}) \to {E}^m(D^p\times S^q)$ is defined as follows.
Represent $f\in \pi_{q}(V_{m-q,p})$ as a smooth map $f\colon D^p\times S^q\to D^{m-q}$ linear in each fiber $D^p\times y$, $y\in S^q$.
Define $\tau(f)$ to be the composition $D^p\times S^q\to D^{m-q}\times S^q\to S^m$ of the embedding $f\times \mathrm{projection}$ and the standard
embedding $s\colon D^{m-q}\times S^q\to S^m$, i.e., $\tau(f)(x,y)=s(f(x,y),y)$ for each $x\in D^p$, $y\in S^q$.

{\it The exactness at the terms ${E}^m(D^p\times S^q)$ and $E^{m}(S^q)$} is checked directly.

{\it Proof of the exactness at the term $\pi_{q}(V_{m-q,p})$.} Let $f\colon S^{q+1}\to S^{m+1}$ be an embedding. Then $f$ is isotopic to a {\it standardized} embedding $f'\colon S^{q+1}\to S^{m+1}$, i.e., satisfying the conditions:
\begin{itemize}
\item $f'\colon D^{q+1}_-\to D^{m+1}_-$ is the abbreviation of the standard embedding $S^{q+1}\to S^{m+1}$;
\item $f'(\mathrm{Int}D^{q+1}_+)\subset \mathrm{Int}D^{m+1}_+$.
\end{itemize}
Take a $p$-framing of the disc $f'(D^{q+1}_+)$.
Represent this framing by an embedding $g\colon D^p\times D^{q+1}_+\to D^{m+1}_+$ linear in each fiber $D^p\times y$, $y\in D^{q+1}_+$. Clearly, $\tau Ob (f')=g\left|_{D^p\times \partial D^{q+1}_+}\right.$.
Thus the embedding $\tau Ob(f')\colon D^p\times \partial D^{q+1}_+\to \partial D^{m+1}_+$ extends to the embedding $g\colon D^p\times D^{q+1}_+\to D^{m+1}_+$.
So $\tau Ob (f')$ is isotopic to the standard embedding $D^p\times S^q\to S^m$, cf.~Lemma~\ref{l2} above. Thus $\image Ob\subset \kernel \tau$. Analogously $\image Ob\supset \kernel\tau$.
\end{proof}

\section{The exact sequence} \label{sect2}


First let us define the homomorphisms in Theorem~\ref{th3}. These maps are well-defined for  $m>p+q+2$ and are homomorphisms for $m>2p+q+2$.

Throughout this section we replace the group $E^m(D^p\times S^q)$ in Theorem~\ref{th3} by the group $E^m(D^p_-\times S^q)$. The former and the latter groups are identified, say, by the isomorphism induced by the central projection $D^p_-\to D^p$ from the point $(0,\dots,0,1)\in\mathbb{R}^{p+1}$.

The map $i^*\colon {{E}^m(S^p\times S^q)}\to {{E}^m(D^p_-\times S^q)}$ is restriction-induced. It follows directly by the construction of the group structure that the map is a homomorphism for $m>2p+q+2$.

\subsection*{Definition of $\sigma^*$.}
\addcontentsline{subsection}{Definition of sigma*}{Definition of sigma*}
The map $\sigma^*\colon{{E}_0^{m}(S^{p+q}\sqcup S^{q})}\to{{E}^{m}(S^{p}\times S^{q})}$ is defined as follows; see Figure~\ref{fig6}.
Represent an element of ${{E}_0^{m}(S^{p+q}\sqcup S^{q})}$ by an embedding $f\colon S^{p+q}\sqcup S^q\to S^m$ such that the restriction  $f\colon S^q\to S^m$ is standard and $f(S^{p+q})\cap s(D^{p+1}\times S^q)=\emptyset$. The latter condition can be always achieved because $f(S^{p+q})$ can be moved aside a neighbourhood of $f(S^q)=s(0\times S^q)$ by an appropriate isotopy.
Join the images $s(S^p\times S^q)$ and $f(S^{p+q})$
by an arc, whose interior misses these images.
Let $\sigma^*(f)\colon S^p\times S^q\to S^m$ be the embedded connected sum of $s\colon S^p\times S^q\to S^m$ and $f\colon S^{p+q}\to S^m$ along this arc.
For $p=0$ or $q=0$ the manifold $S^p\times S^q$ is disconnected, thus we need to specify that the endpoints of the above arc belong to $s(D^p_+\times D^q_+)$ and $fD^{p+q}_+$.


\begin{figure}[htb]
\includegraphics[width=17cm]{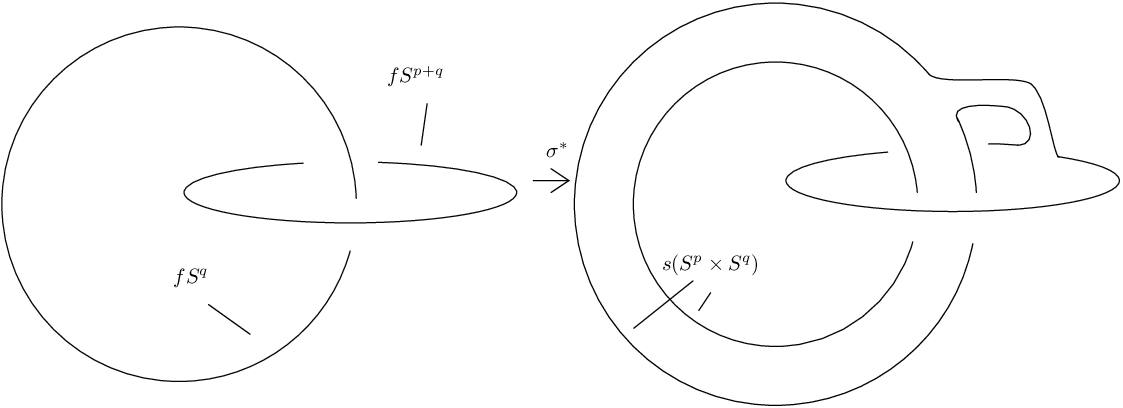}
\caption{Definition of the map $\sigma^*\colon {{E}_0^{m}(S^{p+q}\sqcup S^{q})}\to{{E}^{m}(S^{p}\times S^{q})}$ ($p=0, q=1, m=3$)}
\label{fig6}
\end{figure}

To show that the map $\sigma^*$ is well-defined, we need the following result due to A.~Haefliger \cite[Proof of Theorem~7.1]{Hae66C}.  For convenience of the reader we present the proof obtained in a discussion with A.~Skopenkov and A.~Zhubr.

\begin{lemma}\label{l-haefliger}
If two embeddings $S^{p+q}\sqcup S^q\to S^m$ with standard restrictions to $S^q$ are isotopic, then there is an isotopy between them fixed on $S^q$.
\end{lemma}

\begin{proof}[Proof] 
Take two isotopic embeddings $f_1,f_2\colon S^{p+q}\sqcup S^q\to S^m$ whose restrictions to $S^q$ are standard. Move the images $f_1(S^{p+q}),f_2(S^{p+q})$ aside $D^m_-$ along the great half-circles passing through the points $(0,\dots,0,\pm 1)\in\mathbb{R}^{m+1}$ and having no other common points with $s(S^q)$. This isotopy allows to assume that
$f_1(S^{p+q}),f_2(S^{p+q})\subset \interior D^m_+$.

Since $f_1$ and $f_2$ are isotopic it follows that there is an orientation-preserving diffeomorphism $h\colon S^m\to S^m$ fixed on $s(S^q)$
such that $hf_1=f_2$ on $S^{p+q}$. 
Then both the standard embedding $s\colon \mathbb{R}^{m-q}\times S^q\to S^m$
and the composition $hs\colon \mathbb{R}^{m-q}\times S^q\to S^m$ are tubular neighborhoods of $s(S^q)$ in $S^m$.
By the uniqueness of tubular neighborhoods \cite[Theorem~5.5 in Chapter~ 4]{Hir76} it follows that there is an isotopy $h_t\colon \mathbb{R}^{m-q}\times S^q\to S^m$ fixed on $s(0\times S^q)$ such that $h_0=s$, $h_1(\mathbb{R}^{m-q}\times S^q)=hs(\mathbb{R}^{m-q}\times S^q)$, and $(hs)^{-1}h_1$ is an automorphism of the trivial bundle $\mathbb{R}^{m-q}\times S^q\to S^q$.

The latter automorphism can be made identical on the contractible subset $\mathbb{R}^{m-q}\times D^q_+$ by an appropriate isotopy because automorphisms of the bundle $\mathbb{R}^{m-q}\times S^q\to S^q$ up to isotopy are in bijection with smooth maps $S^{q}\to V_{m-q,m-q}$ up to smooth homotopy. Thus we may assume that $h_1=hs$ on $\mathbb{R}^{m-q}\times D^q_+$.


Then $h_ts^{-1}f_1\colon S^{p+q}\sqcup S^q\to S^m$ is the required isotopy between $f_1$ and $f_2$.
Indeed, the assumption $f_1(S^{p+q}),f_2(S^{p+q})\subset \interior D^m_{+}= s(\mathbb{R}^{m-q}\times \interior D^q_+)$ implies that $h_ts^{-1}f_1$ is well-defined, $h_0s^{-1}f_1=ss^{-1}f_1=f_1$ and $h_1s^{-1}f_1=hf_1=f_2$ on $S^{p+q}$.
Finally, since
$h_t\colon \mathbb{R}^{m-q}\times S^q\to S^m$ is fixed on $s(0\times S^q)$ it follows that $h_ts^{-1}f_1=h_t(0\times \mathrm{Id})=s$ on ${S^q}$.
\end{proof}

\begin{lemma} \label{clwelldef} 
Assume that $m>p+q+2$. Then the map
$\sigma^*\colon {{E}_0^{m}(S^{p+q}\sqcup S^{q})}\to{{E}^{m}(S^{p}\times S^{q})}$
is well-defined by the above construction.
\end{lemma}

\begin{proof}
Let us show that the isotopy class $\sigma^*(f)$ depends neither on the choice of a particular representative $f\colon S^{p+q}\sqcup S^q\to S^m$ within an isotopy class nor on the choice of arc joining $s(S^p\times S^q)$ and $fS^{p+q}$. Take two isotopic embeddings $f_1,f_2\colon S^{p+q}\sqcup S^q\to S^m$ whose restrictions to $S^q$ are standard such that $f_1(S^{p+q})\cap s(D^{p+1}\times S^q)=f_2(S^{p+q})\cap s(D^{p+1}\times S^q)=\emptyset$. Take two arcs $x_1y_1$ and $x_2y_2$ joining $s(S^p\times S^q)$ with $f_1S^{p+q}$ and $f_2S^{p+q}$ respectively.

By Lemma~\ref{l-haefliger} there is an isotopy $f_t$ between $f_1$ and $f_2$ fixed on $S^q$. We may assume that $f_tS^{p+q}\cap s(D^{p+1}\times S^q)=\emptyset$ by uniformly moving $f_tS^{p+q}$ aside a neighbourhood of $fS^q$.

Join the two arcs $x_1y_1$ and $x_2y_2$ by a general position family of arcs $x_ty_t$ with the endpoints at $s(S^p\times S^q)$ and $f_tS^{p+q}$. Since $m>p+q+2$,  by general position it follows that the interior of each arc $x_ty_t$ misses $f_t S^{p+q}$ and $s(D^{p+1}\times S^q)$. Let $\sigma^*(f_t)\colon S^p\times S^q\to S^m$ be the embedding obtained by the connected summation of $s$ and $f_t$ along the arc $x_ty_t$. Strictly speaking, the embedding $\sigma^*(f_t)$ is not uniquely determined by the arc $x_ty_t$ but depends on the particular way of making connected summation along the arc. But clearly we can choose $\sigma^*(f_t)\colon S^p\times S^q\to S^m$ so that it smoothly depends on $t$.
Then $\sigma^*(f_t)$ is an isotopy between $\sigma^*(f_1)$ and $\sigma^*(f_2)$.
That is, the isotopy classes of $\sigma^*(f_1)$ and $\sigma^*(f_2)$ are the same.
\end{proof}

\begin{lemma} \label{sigma-homomorphism}
Assume that $m>2p+q+2$. Then the map $\sigma^*\colon {{E}_0^{m}(S^{p+q}\sqcup S^{q})}\to{{E}^{m}(S^{p}\times S^{q})}$ is a homomorphism.
\end{lemma}

\begin{proof} Represent two elements of ${{E}_0^{m}(S^{p+q}\sqcup S^{q})}$ by two embeddings $f_1,f_2\colon S^{p+q}\sqcup S^{q}\to S^m$ whose restrictions to $S^q$ are standard such that
$f_1(S^{p+q})\subset D^m_+-s(D^{p+1}\times S^q)$ and $f_2(S^{p+q})\subset D^m_--s(D^{p+1}\times S^q)$.

Join the three images $s(S^p\times S^q)$, $f_1S^{p+q}$, $f_2S^{p+q}$  pairwise by general position arcs $x_1y_1$, $z_1z_2$, $y_2x_2$ in cyclic order. Then the embedding $\sigma^*(f_1+f_2)$ is the embedded connected sum of the embeddings $f_1$, $f_2$, $s$ along the arcs $z_1z_2$ and $y_2x_2$. The embedding $\sigma^*(f_1)+\sigma^*(f_2)$ the embedded connected sum of the embeddings $f_2$, $s$, $f_1$ along the arcs $y_2x_2$ and $x_1y_1$. Joining $z_1z_2$ and $x_1y_1$ by a family of arcs analogously to the last paragraph of the proof of Lemma~\ref{clwelldef} we obtain that $\sigma^*(f_1+f_2)$ and $\sigma^*(f_1)+\sigma^*(f_2)$ are isotopic. That is, $\sigma^*(f_1+f_2)=\sigma^*(f_1)+\sigma^*(f_2)$ as isotopy classes.
\end{proof}

\begin{remark*} The group $E^{m}(S^{p+q})$ can be identified with the subgroup of the group ${{E}_0^{m}(S^{p+q}\sqcup S^{q})}$ formed by all the embeddings with unlinked components. 
If one makes such an identification then the map $\sigma^*\colon {{E}_0^{m}(S^{p+q}\sqcup S^{q})}\to{{E}^{m}(S^{p}\times S^{q})}$ extends the map $\sigma^*\colon {{E}^{m}(S^{p+q})}\to{{E}^{m}(S^{p}\times S^{q})}$ defined in \S\ref{sectprem}.
\end{remark*}

Now let us give two equivalent definitions of the map $\partial^*\colon  {{E}^m(D^p_-\times S^q)}\to {{E}^{m-1}_0(S^{p+q-1}\sqcup S^{q-1})}$. The first definition is easier to understand and to use in Section~\ref{sect3}, while the second one is easier to use in the proof of the exactness.

\subsection*{First definition of $\partial^*$}
\addcontentsline{subsection}{First definition of delta*}{First definition of delta*}

The map $\partial^*\colon  {{E}^m(D^p_-\times S^q)}\to {{E}^{m-1}_0(S^{p+q-1}\sqcup S^{q-1})}$ is the composition
\begin{equation*}
\xymatrix@1{
{{E}^m(D^p_-\times S^q)}          \ar[r]^-{Ob}   &
{\pi_{q-1}(S^{m-p-q-1})}        \ar[r]^-{Ze}   &
{{E}_0^{m-1}(S^{p+q-1}\sqcup S^{q-1}),}
}
\end{equation*}
where the maps $Ob$ and $Ze$ are defined as follows.


The map $Ob\colon {{E}^m(D^p_-\times S^q)}\to \pi_{q-1}(S^{m-p-q-1})$ is { the obstruction to the existence of a unit normal vector field} on the embedded manifold $D^p_-\times S^q$. To define the map, take an embedding $f\colon D^p_-\times S^q\to S^m$.
Take a (unique up to homotopy) trivialization of the normal bundle $\nu$ to the disc $f(D^p_-\times D^q_+)$ in $S^m$. Take a (unique up to homotopy) unit normal vector field on $f(D^p_-\times D^q_-)$. Thus the sphere $f(x\times S^{q-1})$, where $x\in D^p_-$ is fixed, is equipped both with a normal vector field and with a trivialization of the restriction of the normal bundle $\nu$. Using the trivialization identify each fiber of the normal bundle $\nu$ with the space $\mathbb{R}^{m-p-q}$. To each point $y\in S^{q-1}$ assign the unit vector of the field at the point $f(x,y)$. Thus a map $S^{q-1}\to S^{m-p-q-1}$ is defined. By definition, $Ob(f)\in \pi_{q-1}(S^{m-p-q-1})$ is the homotopy class of this map.

The map $Ze\colon \pi_{q-1}(S^{m-p-q-1})\to{{E}_0^{m-1}(S^{p+q-1}\sqcup S^{q-1})}$ is {\it the Zeeman construction} of the link with a given linking number. To define the map, consider the nested
standard embeddings $s\colon S^{q-1}\to S^{m-1}$ and $s\colon S^{p+q-1}\to S^{m-1}$. Take a trivialization of the normal bundle to the disc $s(D^{p+q-1}_+)$ in $S^{m-1}$. Analogously to the above each element $f\in \pi_{q-1}(S^{m-p-q-1})$ determines (up to homotopy) a unit vector field on $s(S^{q-1})$ normal to $s(S^{p+q-1})$. Push the sphere $s(S^{q-1})$ in the direction of this vector field. By definition, $Ze(f)$ is the link formed by the obtained embedding $S^{q-1}\to S^{m-1}$ and the standard embedding $s\colon S^{p+q-1}\to S^{m-1}$.

\begin{lemma} \label{delta-homomorphism} Assume $m>p+q+2$.
Then the map $\partial^*\colon  {{E}^m(D^p_-\times S^q)}\to {{E}_0^{m-1}(S^{p+q-1}\sqcup S^{q-1})}$ is a homomorphism.
\end{lemma}

\begin{proof} 
It suffices to show that both maps ${Ob}$ and ${Ze}$ are homomorphisms.

To prove that ${Ob}$ is a homomorphism, take two embeddings $f_1,f_2\colon D^p_-\times S^q\to S^m$.
Clearly, $f_1$ is isotopic to an embedding $f_1'$ satisfying the following properties; cf. Lemma~\ref{l1}:
\begin{itemize}
\item $f'_1\colon D^p_-\times (S^q-D^q_{++})\to S^m-D^m_{++}$ is the restriction of the standard embedding $s\colon S^p\times S^q\to S^m$;
\item $f'_1(D^p_-\times \mathrm{Int}D^q_{++})\subset \mathrm{Int}D^m_{++}$.
\end{itemize}
Analogously, $f_2$ is isotopic to an embedding $f_2'$ satisfying the same properties with $D^q_{++}$ and $D^m_{++}$ replaced by $r_qD^q_{++}$ and $r_mD^m_{++}$ respectively. The sum $f_1+f_2$ can be represented by the embedding $f\colon D^p_-\times S^q\to S^m$ given by the formula; cf. Lemma~\ref{l3}:
$$f(x,y)=
\begin{cases} f_1'(x,y),& y\in D^q_{+};\\ r_mr_{m+1} f_2'(x, r_qr_{q+1}y),& y\in D^q_-.\end{cases}
$$
Let $u$ be an $(m-p-q)$-framing of $s(D^p_-\times S^q)$. The first vector field of the framing $u$ determines a vector field $v$ on $s(D^p_-\times \partial D^{q-1}_{+})$. Extend $v$ to a normal vector field on $f(D^p_-\times (D^q_{++}\cup r_{q+1}D^q_{++}))$. The restrictions of the constructed field to the spheres $s(x\times \partial (D^q_{++}\cup r_{q+1}D^q_{++}))$, $s(x\times \partial D^q_{++})$, $s(x\times r_{q+1}\partial D^q_{++})$ together with the  framing $u$ determine the elements ${Ob}(f_1+f_2)$, ${Ob}(f_1)$, ${Ob}(f_2)$ respectively.
Thus ${Ob}(f_1+f_2)={Ob}(f_1)+{Ob}(f_2)$.

It is also easy to prove that ${Ze}$ is a homomorphism, cf. \cite[Theorem~10.1]{Hae66C}.
\end{proof}


\subsection*{Second definition of $\partial^*$.}
\addcontentsline{subsection}{Second definition of delta*}{Second definition of delta*}

The map $\partial^*\colon  {{E}^m(D^p_-\times S^q)}\to {{E}^{m-1}_0(S^{p+q-1}\sqcup S^{q-1})}$ is an obstruction to extend an embedding
$D^p_-\times S^q\to S^m$ to an embedding $S^p\times S^q\to S^m$. 
To define this obstruction we need two lemmas and an auxiliary definition. Recall that $B\subset \interior D^p_+\times \interior D^q_+$ is a fixed $(p+q)$-dimensional ball.

\begin{figure}[thbp]
\begin{tabular}{c}
\includegraphics[width=12cm]{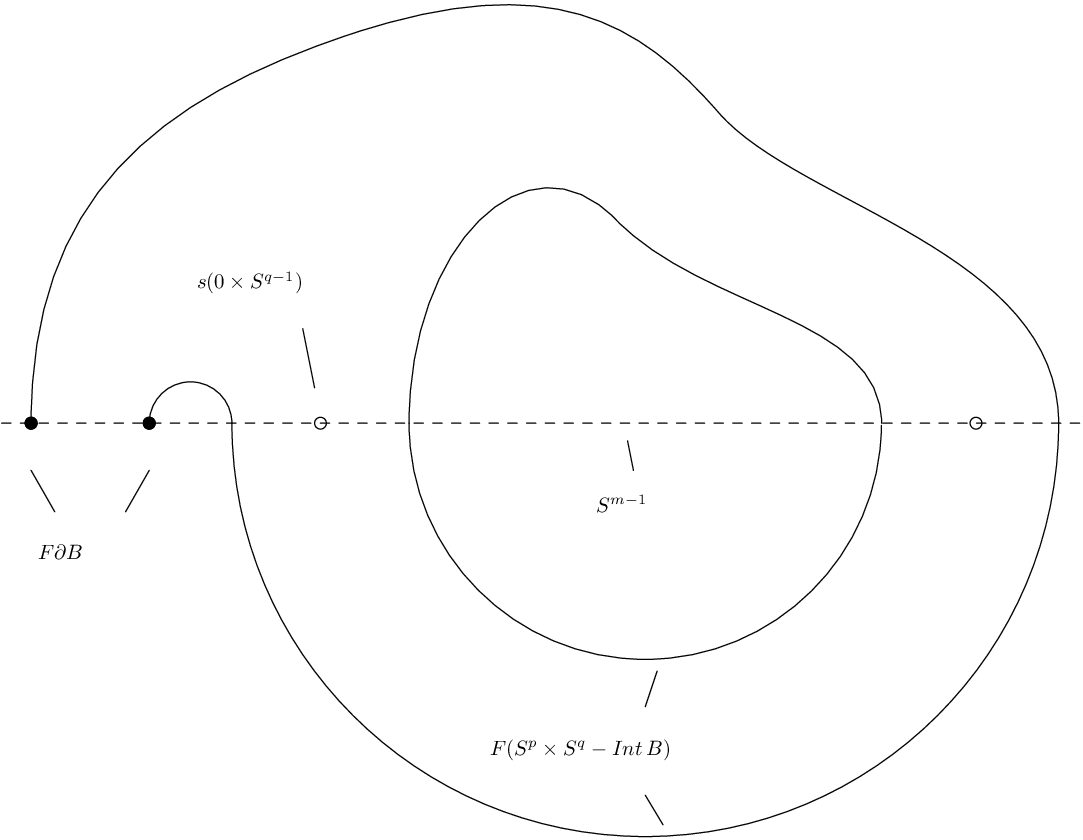}\\[10pt]
$\downarrow \partial^*$\\[10pt]
\includegraphics[width=12cm]{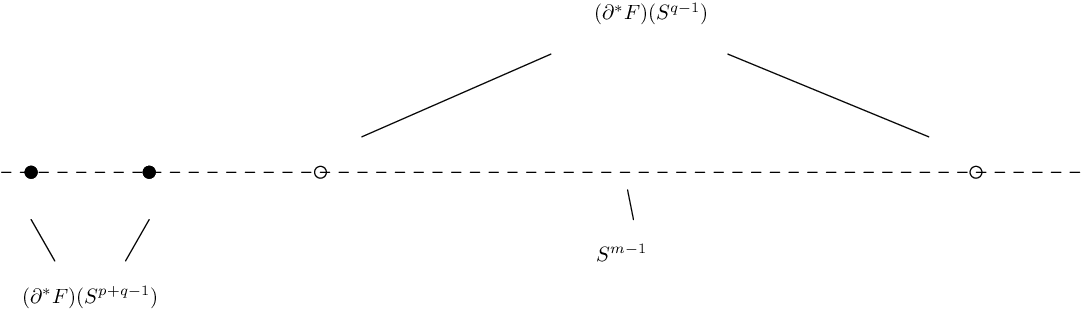}
\end{tabular}
\caption{The second definition of the map $\partial^*$ ($p=0, q=1, m=2$)}
\label{altdelta}
\end{figure}

%


\begin{lemma} \label{easy-extension} Assume that $m> 2p+q+2$. Then

\textup{(a)} any embedding $D^p_-\times S^q\to S^m$ extends to an embedding $S^{p}\times S^q-\interior B\to S^m$;

\textup{(b)} any two embeddings $S^{p}\times S^q-\interior B\to S^m$, whose restrictions to $D^p_-\times S^q$ are isotopic, are also isotopic.
\end{lemma}

\begin{proof} (a) Let $f\colon D^p_-\times S^q\to S^m$ be an embedding. Take a point $y\in D^q_-$.
Extend the embedding $f\colon D^p_-\times y\to S^m$ to a general position smooth map $g\colon S^p\times y\to S^m$.
Since $m> 2p+q+2$ it follows that the map $g$ does not have self-intersections and  $g(\mathrm{Int}D^p_+\times y)\cap f(D^p_-\times S^q)=\emptyset$. 

The restriction $f\colon \partial D^p_-\times D^q_-\to S^m$ defines a $q$-framing of the sphere $f(\partial D^p_-\times y)$.
The complete obstruction to extension of this $q$-framing to a $q$-framing of the disc $g(D^p_+\times y)$ belongs to the group
$\pi_{p-1}(V_{m-p,q})$. Since $m>2p+q+2$ it follows that the latter group vanishes. Thus the $q$-framing of the sphere $f(\partial D^p_-\times y)$ extends to a $q$-framing of the disc $g(D^p_+\times y)$. The latter $q$-framing defines an embedding
$D^p_-\times S^q\cup D^p_+\times D^q_-\to S^m$ extending the embedding $f\colon D^p_-\times S^q\to S^m$.
Clearly, the obtained embedding $D^p_-\times S^q\cup D^p_+\times D^q_-\to S^m$ extends to an embedding $S^{p}\times S^q-\interior B\to S^m$.

(b) This is a relative version of the argument from point (a).
\end{proof}


\begin{definition*} (See Figure~\ref{altdelta} to the top.)
An embedding $F\colon S^p\times S^q-\interior B\to S^m$ is {\it $B$-standardized} if
\begin{enumerate}
\item[(1)] $F\colon S^p\times D^q_-\to D^m_-$ is the standard embedding;

\item[(2)] $F(S^p\times\interior D^q_+-B)\subset\interior D^m_+$;

\item[(3)] $F(\partial B)\subset \partial D^{m}_-$; and

\item[(4)] $F(\partial B)\cap s(D^{p+1}\times S^{q-1})=\emptyset$.
\end{enumerate}

An isotopy $F_t\colon S^p\times S^q-\interior B\to S^m$ is {\it $B$-standardized} if for each $t\in I$
the embedding $F_t$ is $B$-standardized. A {\it $B$-standardized} embedding $F\colon S^p\times S^q\to S^m$ is defined analogously,
only the above properties~(3) and~(4) are replaced by

\begin{enumerate}
\item[(3')] $F(\interior B)\subset \interior D^m_-$;
\item[(4')] $F(\interior B)\cap s(D^{p+1}\times D^{q}_-)=\emptyset$.
\end{enumerate}
\end{definition*}

\begin{lemma}\label{lura} Assume that $m> 2p+q+2$. Then

\textup{(a)} any embedding $S^p\times S^q-\interior B\to S^m$ is
isotopic to a $B$-standardized embedding;

\textup{(b)} any embedding $S^p\times S^q\to S^m$ is
isotopic to a $B$-standardized embedding; and

\textup{(c)} any isotopy between $B$-standardized embeddings $S^p\times S^q-\interior B\to S^m$
is isotopic relative to the ends to a $B$-standardized isotopy.
\end{lemma}

\begin{proof} (a) Take an embedding $F\colon S^p\times S^q-\interior B\to S^m$. By a generalization of Lemma~\ref{l1}.a (Lemma~\ref{lstand} below) $F$ is isotopic to an embedding $F'\colon S^p\times S^q-\interior B\to S^m$ satisfying properties~(1) and~(2) of a $B$-standardized embedding.
Assume in addition that 
$F'=s$ on $S^p\times (S^q-D^q_{+\epsilon})$ and $F'(S^p\times D^q_{+\epsilon}-B)\subset D^m_{+\epsilon}$ for some $\epsilon>0$ such that $B\subset D^p_+\times D^q_{+\epsilon}$.

In order to achieve properties~(3) and~(4) we argue by the following plan. First we construct a neighborhood $\bar D^m_+\subset D^m_+$ of (certain smoothing of) the disc $F'(S^p\times D^q_{+\epsilon}-\interior B)\cup s(D^{p+1}\times \partial D^q_{+\epsilon})$ such that $s(D^{p+1}\times S^{q-1})\subset\partial\bar D^m_+$; see Figure~\ref{deltafin}. Then we move $\bar D^m_+$ to $D^m_+$ by an isotopy fixed on $s(D^{p+1}\times S^{q-1})$. This isotopy takes Figure~\ref{deltafin} to the top part of Figure~\ref{altdelta}. In particular, the isotopy transforms $F'\colon S^p\times S^q-\interior B\to S^m$ to a $B$-standardized embedding.

\begin{figure}[htb]
\includegraphics[width=14cm]{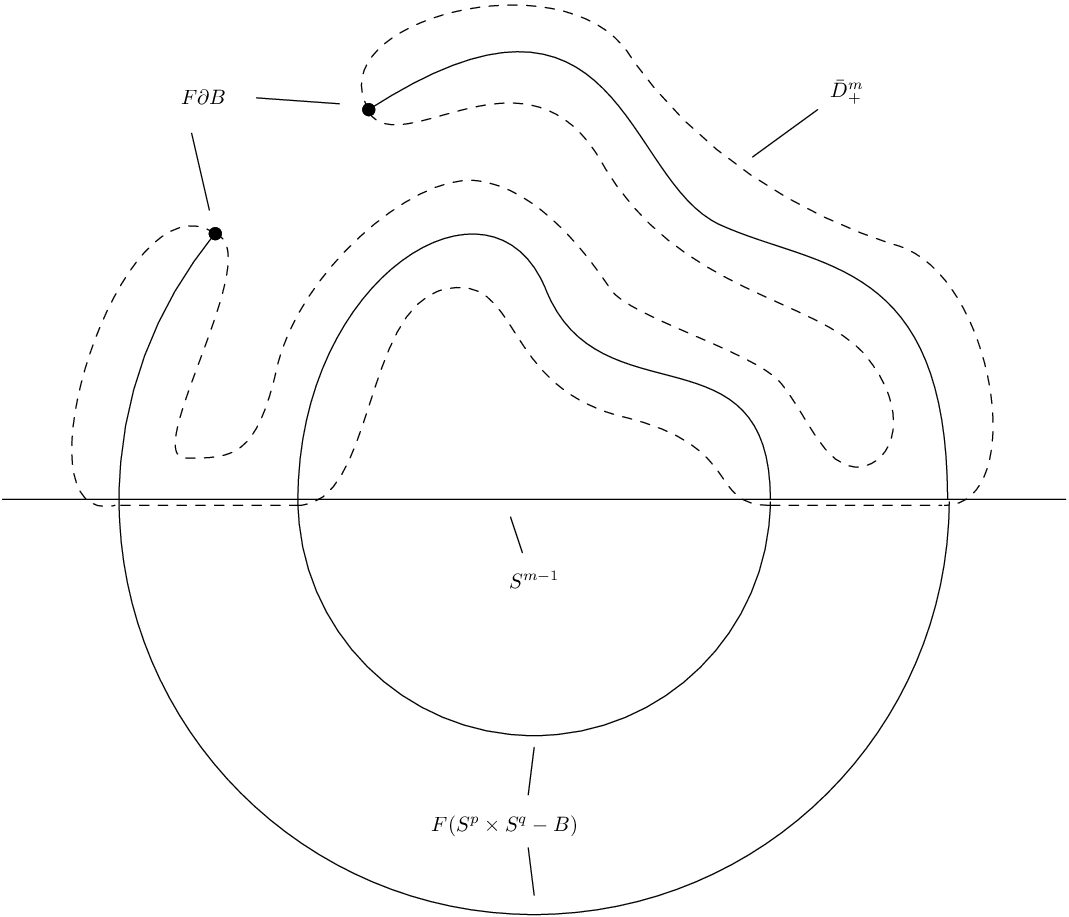}
\caption{Making an embedding $(S^0\times S^1-B)\to S^2$
$B$-standardized 
}
\label{deltafin}
\end{figure}

To construct the neighborhood $\bar D^m_+$ perform the following version of the standard embedded surgery over $F'$ (slightly different from the version defined in \S2).
Fix a diffeomorphism $D^{p+q}_+\cong (S^p\times D^q_{+\epsilon}-\mathrm{Int B})\cup D^{p+1}\times S^{q-1}$.
Denote $s'=\alpha_{\epsilon/2}s\colon D^{p+1}_{1-\epsilon}\times S^{q-1}\to \partial D^m_{+\epsilon/2}$. Extend it to an embedding
$s'\colon D^{p+1}\times S^{q-1}\to s(D^{p+1}\times (\interior D^q_{+}-\interior D^q_{+\epsilon}))$ smoothly blending it with the standard embedding $s\colon S^p\times D^q_{+\epsilon}\to D^m_{+\epsilon}$.
Gluing together the embeddings $s'\colon D^{p+1}\times S^{q-1}\to
\interior D^m_+-\interior D^m_{+\epsilon}$
and $F'\colon S^p\times D^q_{+\epsilon}-\interior B\to D^m_{+\epsilon}$,
we get an embedding $G\colon D^{p+q}_+\to D^m_+$. By definition, this is the result of the version of {\it the standard surgery}.

To proceed, we make the following convention. Denote by $D^{m-q}\Breve{\times} D_+^{q}\subset \mathbb{R}^{m-q}{\times} D_+^{q}$ the $C^1$-smooth disc given by the inequality
$x_1^2+\dots+x_{m-q}^2\le (y_{q+1}+1)^2$, where $(x_1,\dots,x_{m-q})\in \mathbb{R}^{m-q}$ and $(y_1,\dots,y_{q+1})\in D_+^{q}$. By \emph{a closed tubular neighborhood} of a $q$-dimensional disc embedded into $S^m$ we mean the restriction of a tubular neighborhood $\mathbb{R}^{m-q}{\times} D_+^{q}\to S^m$ to the disc $D^{m-q}\breve{\times} D_+^{q}$. This guarantees that
the image of a closed tubular neighborhood is a $C^1$-smooth disc.

Let $\bar D^m_+$ be the image of a closed tubular neighborhood of the $(p+q)$-dimensional disc $G(D^{p+q}_+)$ in $D^m_+$. Then $F'(\partial B)\subset \partial \bar D^m_+$. Since
$G\colon D^{p+1}\times S^{q-1}\to s(D^{p+1}\times(\interior D^q_+-\interior D^q_{+\epsilon}))$ is close to $s\colon D^{p+1}\times S^{q-1}\to S^{m-1}$ we may assume that $F(S^p\times (\interior D^q_+-\interior D^q_{+\epsilon}))\subset \interior \bar D^m_+$ and
$\partial \bar D^m_+\cap \partial D^m_+$ is a the image of a closed tubular neighborhood of $s(D^{p+1}_{1+\epsilon}\times S^{q-1})$ in $\partial D^m_+$. In particular, $F'\colon S^p\times D^q_+-\interior B\to\bar D^m_+$ is a proper embedding.

Let us construct an isotopy fixed on $s(D^{p+1}\times S^{q-1})$ moving $\bar D^m_+$ to $D^m_+$.
By the unkotting theorem moving the boundary the disc $F(x\times D^q_+)$, where $x=(0,\dots,0,1)\in D^p_-$, is properly knotted neither in $D^m_+$ nor in $\bar D^m_+$. Thus both $D^m_+$ and $\bar D^m_+$ are the images of some closed tubular neighborhoods
$g,\bar g\colon D^{m-q}\breve{\times} D_+^{q}\to D^m_+$
of the disc $F(x\times D^q_+)$.

We may assume that
$g(D^{m-q}{\times} S^{q-1})=\bar g(D^{m-q}{\times} S^{q-1})=\partial \bar D^m_+\cap \partial D^m_+$.
Indeed, say, the restriction $g\colon D^{m-q}{\times} S^{q-1}\to S^{m-1}$ is a closed tubular neighborhood of $F(x\times S^{q-1})$ in $\partial D^m_+$. The intersection $\partial \bar D^m_+\cap \partial D^m_+$ is also the image of a closed tubular neighborhood of $s(D^{p+1}_{1+\epsilon}\times S^{q-1})$ and hence of $F(x\times S^{q-1})$. By the uniqueness of closed tubular neighborhoods \cite[Theorem~6.5 in Chapter~4]{Hir76} there is an isotopy
$\bar h_t\colon D^{m-q}{\times} S^{q-1}\to \partial D^m_+$
such that $\bar h_0(D^{m-q}{\times} S^{q-1})=\partial \bar D^m_+\cap \partial D^m_+$ and $\bar h_1(D^{m-q}{\times} S^{q-1})=g(D^{m-q}{\times} S^{q-1})$. By the isotopy extension theorem \cite[Theorem~1.3 in Chapter~8]{Hir76} the isotopy
$\bar h_t\bar h_0^{-1}(D^{m-q}{\times} S^{q-1})=\partial \bar D^m_+\cap \partial D^m_+\to\partial D_+^m$
extends to an ambient isotopy $\hat h_t\colon D^m_+\to D^m_+$ such that $\hat h_0=\mathrm{Id}$ on $D^m_+$. By the isotopy extension theorem fixed on the boundary we may assume that $\hat h_1=\mathrm{Id}$ on $F(x\times D^q_+)$.
Then $\hat h_1g\colon D^{m-q}\breve{\times} D_+^{q}\to D^m_+$ is a closed tubular neighborhood of $F(x\times D^q_+)$ satisfying $\hat h_1g(D^{m-q}{\times} S^{q-1})=\partial \bar D^m_+\cap \partial D^m_+$. 

Now by the uniqueness of closed tubular neighborhoods 
there is an isotopy $\bar H_t\colon D^{m-q}\breve{\times} D_+^{q}\to D^m_+$
such that $\bar H_0(D^{m-q}\breve{\times} D_+^{q})=g(D^{m-q}\breve{\times} D_+^{q})=\bar D^m_+$, $\bar H_1(D^{m-q}\breve{\times} D_+^{q})=\bar g(D^{m-q}\breve{\times} D_+^{q})=D_+^{m}$, and
$\bar H_t(D^{m-q}{\times} S^{q-1})\subset \partial D^m_+$.
By the isotopy extension theorem 
the isotopy $\bar H_t\bar H_0^{-1}\colon \partial \bar D^m_+\cap\partial D^m_+\to\partial D^m_+$ extends to an ambient isotopy $\hat H_t\colon D^m_+\to D^m_+$ such that $\hat H_0=\mathrm{Id}$. The composition $H_t=\hat H_t^{-1}\bar H_t\bar H_0^{-1}\colon \bar D^m_+\to D^m_+$ is the required isotopy fixed on $s(D^{p+1}\times S^{q-1})$ moving $\bar D^m_+$ to $D^m_+$.
Indeed, $H_0(\bar D^m_+)=\hat H_0^{-1}\bar H_0\bar H_0^{-1}(\bar D^m_+)=\bar D^m_+$,  $H_1(\bar D^m_+)=\hat H_1^{-1}\bar H_1\bar H_0^{-1}(\bar D^m_+)=D^m_+$, and $H_t$ is fixed on $s(D^{p+1}\times S^{q-1})$ because $H_t=\hat H_t^{-1}\bar H_t\bar H_0^{-1}=\hat H_t^{-1}\hat H_t=\mathrm{Id}$ on $\partial \bar D^m_+\cap\partial D^m_+\supset s(D^{p+1}\times S^{q-1})$.


The required isotopy joining the embedding $F'$ with a $B$-standardized embedding is obtained by gluing the embeddings $H_tF'\colon S^p\times D^q_+-\interior B\to D^m_+$ and $s\colon S^p\times D^q_-\to D^m_-$ together.
Since $F'(S^p\times D^q_+-\interior B)\subset \bar D^m_+=H_0(\bar D^m_+)$ it follows that $H_tF'\colon S^p\times D^q_+-\interior B\to D^m_+$ is well-defined. Since $H_t$ is fixed on $s(D^{p+1}\times S^{q-1})$ it follows that the two embeddings agree:
$H_tF'=H_0s=s$ on $S^p\times S^{q-1}$.
Clearly, the gluing satisfies properties (1) and (2) above.
Since $F'(\partial B)\sqcup s(D^{p+1}\times S^{q-1})\subset\partial \bar D^m_+$ it follows that
the embedding $H_1F'\colon S^p\times D^q_+-\interior B\to D^m_+$ satisfies properties (3) and (4).
Assertion~(a) is proved.

(b), (c) These assertions are proved analogously (for (b) there is also  a shorter direct proof).
\end{proof}

Now we are ready to give the second definition of the map $\partial^*$.

\begin{definition*} (See Figure~\ref{altdelta}.)
Take an embedding $f\colon D^p_-\times S^q\to S^m$. Extend it to an embedding $F\colon S^p\times S^{q}-\interior B\to S^{m}$.
Take a $B$-standardized embedding $F'\colon S^p\times S^{q}-\interior B\to S^{m}$ isotopic to $F$.
Set $\partial^*(f)=F'\beta \sqcup s\colon S^{p+q-1}\sqcup S^{q-1}\to S^{m-1}$, where the diffeomorphism $\beta\colon S^{p+q-1}\to\partial B$ is fixed in advance.
\end{definition*}

\begin{lemma} \label{clwelldef2} Assume that $m>2p+q+2$. Then the map
$\partial^*\colon  {{E}^m(D^p_-\times S^q)}\to {{E}_0^{m-1}(S^{p+q-1}\sqcup S^{q-1})}$
is well-defined by this definition.
\end{lemma}

\begin{proof}
The construction of the definition is possible by Lemma~\ref{easy-extension}.a and Lemma~\ref{lura}.a.
The result of the construction does not depend on the choice of the extension $F\colon S^p\times S^{q}-\interior B\to S^{m}$
of the given embedding $f\colon D^p_-\times S^q\to S^m$ by Lemma~\ref{easy-extension}.b. The result does not
depend on the choice of the $B$-standardization $F'\colon S^p\times S^{q}-\interior B\to S^{m}$ by
Lemma~\ref{lura}.c. The result depends only on the isotopy class of the embedding $f\colon D^p_-\times S^q\to S^m$ by
Lemma~\ref{easy-extension}.b and~Lemma~\ref{lura}.c.
\end{proof}

\begin{lemma} \label{cleq} Assume that $m>2p+q+2$. Then the two given definitions of the map
$\partial^*\colon  {{E}^m(D^p_-\times S^q)}\to {{E}_0^{m-1}(S^{p+q-1}\sqcup S^{q-1})}$ are equivalent.
\end{lemma}

This is one of the most technical assertions of this paper. In fact neither the first definition of $\partial^*$ nor this lemma are used in the proof of Theorem~\ref{th3} (except the proof of the assertion that $\partial^*$ is a homomorphism which can also be proved directly). But the first definition is more convenient for applications in the proof of Corollary~\ref{essentials} and in \cite{Sko15}.

\begin{proof}[Proof of Lemma~\ref{cleq}]
For a while denote by $\partial^*_{I}$ and $\partial^*_{II}$ the maps given by the first and the second definitions of $\partial^*$ respectively. We use all the notation from the proof of Lemma~\ref{lura}.

First let us give a geometric construction of $\mathrm{Ob}(f)$.
Let $v$ be the unit vector field on $s'(0\times S^{q-1})$ normal to
$\partial D^m_{+\epsilon}$
and looking towards ${D}^{m}_-$. In particular, the standard embedding $s\colon S^{q-1}\to S^{m-1}$ is the result of pushing
the sphere $s'(0\times S^{q-1})$ in the direction of the field $v$ along the circular arc orthogonal to $S^{m-1}$.
The field $v$ is orthogonal to the disc $G(D^{p+q}_+)$.
Take a framing $u$ of the disc.
The framing $u$ identifies the fibers of the normal bundle to the disc with $\mathbb{R}^{m-p-q}$. Thus the pair $(u,v)$ determines an element of $\pi_{q-1}(S^{m-p-q-1})$ denoted also by $(u,v)$.

Let us show that $(u,v)={Ob} (f)$. Indeed, the vector field $v$ extends to the unit normal field on $s'(D^{p+1}\times S^{q-1})$ tangent to $s(D^{p+1}\times S^{q})\supset s'(D^{p+1}\times S^{q-1})$.
The sphere $s'(x\times S^{q-1})$, where $x=(0,\dots,0,-1)\in\mathbb{R}^{p+1}$, is equipped both with the extended vector field $v$ and the framing $u$. 
This equipment determines an element of $\pi_{q-1}(S^{m-p-q-1})$. This element is $(u,v)$ by the obvious homotopy. On the other hand,  the restriction of the vector field $v$ to the sphere $s'(x\times S^{q-1})=s(x\times \partial D^q_{+\epsilon})$ extends to the field on the disc $s(x\times (S^q-\mathrm{Int} D^q_{+\epsilon}))$, which is normal to $s(S^p\times S^q)$ and tangent to $s(D^{p+1}\times S^q)$. Moreover, the framing $u$ is defined on the disc $F(x\times D^q_{+\epsilon})\subset G(D^{p+q}_+)$. Hence $(u,v)={Ob}(f)$ by definition. 


Let us perform an isotopy putting the link $\partial^*_{II} (f)$ into a convenient position.
The embedding $s'\colon 0\times S^{q-1}\to G(D^{p+q}_+)$ is unknotted because it is isotopic to $s'=\alpha_{\epsilon/2}s\colon x\times S^{q-1}\to G(D^{p+q}_+)$ which extends to the embedding $F'\colon x\times D^q_{+\epsilon}\to G(D^{p+q}_+)$.
Thus applying isotopy extension theorem several times, we get
an ambient isotopy $h_t\colon S^m\to S^m$ such that
$h_1\bar D^m_+=D^{m}_+$,
$h_1G= s$ on $D^{p+q}_+$,
and $h_1s'=s'$ on $0\times S^q$.
By the uniqueness of a closed tubular neighborhood we may assume that the isotopy takes the closed normal tubular neighborhood of $G(D^{p+q}_+)$ with the image $\bar D^m_+$ to a prescribed closed tubular neighborhood of $D^{p+q}_+$ with the image $D^m_+$ up to an automorphism of the normal bundle. Thus we may assume that $h_1$ takes each circular arc starting at $s'(0\times S^{q-1})$ and orthogonal to $S^{m-1}$ to a circular arc orthogonal to $S^{m-1}$.


Let us show that $\partial^*_{II} (f)=Ze(u,v)$.
Clearly, the link $h_1F'\beta\sqcup h_1s\colon  S^{p+q-1}\sqcup S^{q-1}\to S^{m-1}$ is isotopic to $\partial_{II}^*(f)=H_1F'\beta\sqcup s\colon  S^{p+q-1}\sqcup S^{q-1}\to S^{m-1}$ (in spite of that both isotopies $H_t$ and $h_t$ move the sphere $S^{m-1}$). The embedding $h_1F'\beta\colon  S^{p+q-1}\to S^{m-1}$
is standard because it is the restriction of $h_1G$.
The embedding $h_1s\colon S^{q-1}\to S^{m-1}$ is the result of pushing the sphere $s'(0\times S^{q-1})$ in the direction of the vector field $h_1^*v$ along the circular arc orthogonal to $S^{m-1}$. This is the same as pushing the standard embedding $S^{q-1}\to S^{m-1}$ towards the vector field $h_1^*v$.
So, since the isotopy $h_t$ induces a homotopy of the vector field $v$ and the framing $u$ it follows that $\partial^*_{II}(f)={Ze}(h_1^*u,h_1^*v)={Ze}(u,v)=
{Ze}\,{Ob}(f)=\partial^*_{I}(f)$.
\end{proof}

\subsection*{Exactness at $E^m(S^p\times S^q)$}
\begin{proof}[Proof that $\image \sigma^*\subset \kernel i^*$] Represent an element of $E^m_0(S^{p+q}\sqcup S^q)$ by an embedding $f\colon S^{p+q}\sqcup S^q\to S^m$ such that the restriction to $f\left|_{S^q}\right.$ is standard and $fS^{p+q}\cap s(D^{p+1}\times S^q)=\emptyset$.
By definition, the restriction of $\sigma^*(f)$ to $D^p_-\times S^q$ coincides with the restriction of the standard embedding $s\colon S^p\times S^q\to S^m$. Thus $i^*\sigma^*(f)=i^*(s)=i^*(0)=0$.
\end{proof}

\begin{proof}[Proof that $\image \sigma^*\supset \kernel i^*$] Let $f\colon S^p\times S^q\to S^m$ be an embedding whose restriction to $D^p_-\times S^q$ is isotopic to the restriction of the standard embedding. Let us construct an element $g\in E^m_0(S^{p+q}\sqcup S^q)$
such that $\sigma^*(g)=f$.

Let us give the plan of the proof. First we perform an isotopy making the restriction $f\left|_{S^p\times S^q-\interior B}\right.$ standard. A possible result is shown in the first ``frame'' of Figure~\ref{exactness1}. Then we remove the intersection of the image $f(\interior B)$ with $s(D^{p+1}\times S^q)$ by an isotopy of $f$ fixed on  $S^p\times S^q-\interior B$. A possible result is shown in the second ``frame'' of Figure~\ref{exactness1}.
A surgery of the obtained embedding gives the required link $g\colon S^{p+q}\sqcup S^q\to S^m$. The result of the surgery is shown in the third ``frame'' of Figure~\ref{exactness1}. The resulting embedding is of the form
$\sigma^*(g)$ for some embedding $g\colon S^{p+q}\sqcup S^q\to S^m$.

\begin{figure}[htb]
\includegraphics[width=7.9cm]{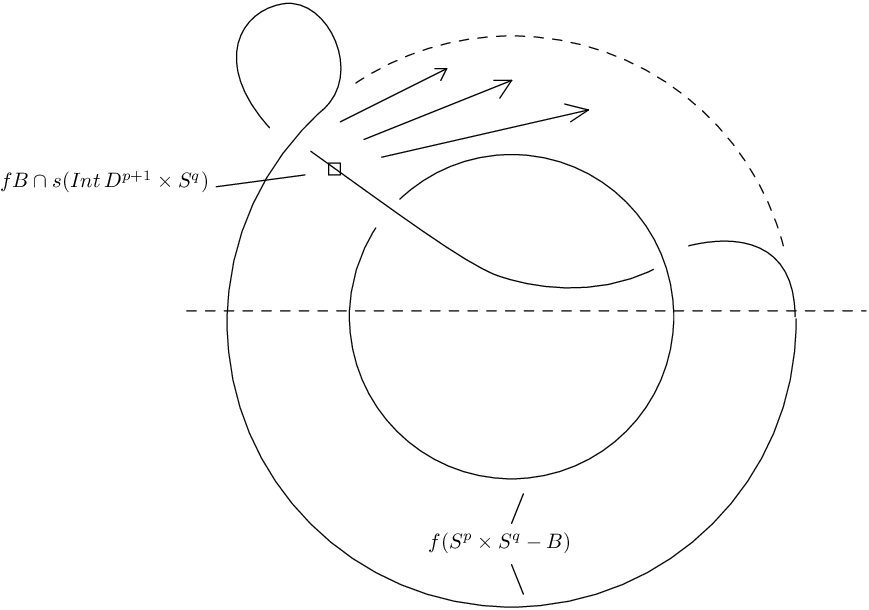}
\begin{tabular}[c]{c}
$\quad\leadsto\quad$\\[5cm]
\end{tabular}
\includegraphics[width=6cm]{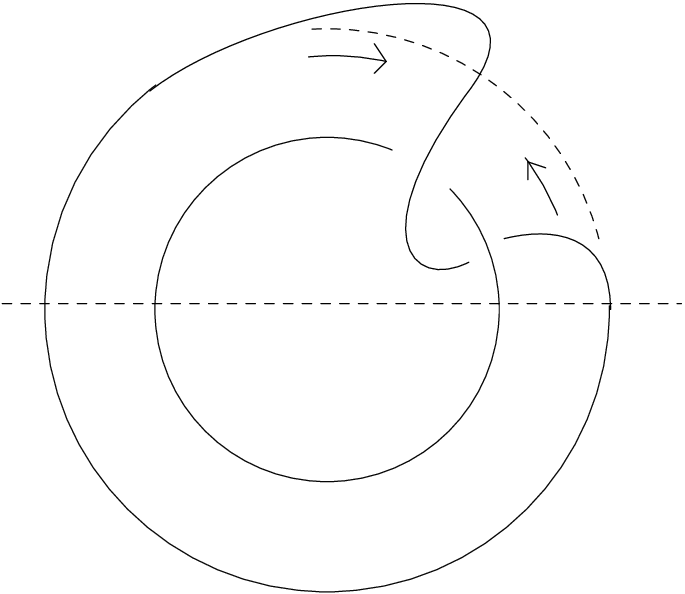}
\begin{tabular}[c]{c}
$\quad\leadsto\quad$\\[5cm]
\end{tabular}
\vspace{-2cm}
\flushleft{$\qquad\qquad$
\includegraphics[width=6.1cm]{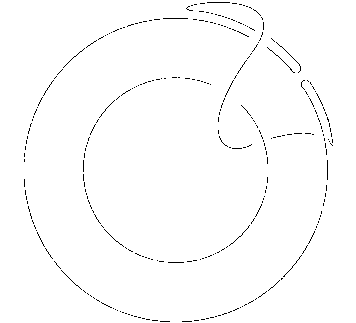}}
\caption{A movie-proof that $\image \sigma^*\supset \kernel i^*$ ($p=0,q=1,m=3$)}
\label{exactness1}
\end{figure}

Let us make the restriction $f\colon S^p\times S^q-\interior B\to S^m$ standard.
By Standardization Lemma~\ref{l1}.a one can make the embedding $f\colon S^p\times S^q\to S^m$ standardized. Thus we may assume that $f$ and $s$ coincide on $S^p\times D^q_-$.
Since the restriction of $f$ to $D^p_-\times S^q$ is isotopic to the standard embedding, by isotopy extension theorem it follows that there is an isotopy of $f$ fixed on $S^p\times D^q_-$ making $f$ standard on $D^p_-\times S^q$. Again by isotopy extension theorem one can make $f$ and $s$ equal also on $S^p\times S^q-\rho(\frac{1}{2}D^{p+q}_-)$, where $\beta\colon D^{p+q}_-\to B$ is a fixed diffeomorphism. Denote the embedding obtained after all the above isotopies still by $f\colon S^p\times S^q\to S^m$.

Let us remove the intersection of $f(\interior B)$ with the image $s(D^{p+1}\times S^q)$.
Since $f$ is standardized it follows that this intersection is a subset of the ball $s(D^{p+1}\times D^q_+)$. 
Take closed tubular neighborhoods of the ball and its face $s(B)$ such that the interior of the union of their images is disjoint with $s(S^p\times S^q-\interior B)$. Take a vector field with the support in the union such that all integral trajectories starting in the ball leave the ball through the face. By \cite[Theorem~1.2 in Chapter~8]{Hir76} there is an ambient isotopy of $S^m$ fixed outside the union moving $f(\interior B)$ along the vector field 
until it becomes disjoint with $s(D^{p+1}\times D^q_+)$. 
Denote the embedding obtained by the isotopy still by $f\colon S^p\times S^q\to S^m$.

Let us perform a surgery over $f$. 
Extend $s\beta\colon D^{p+q}_-\to B\subset S^p\times S^q\to S^m$ to a smooth unknotted embedding $s''\colon S^{p+q}\to S^m$ such that $s''(D^{p+q}_+)\subset s((D^{p+1}-0)\times S^q)$.
Define an embedding $g\colon S^{p+q}\to S^m$ to be the result
of gluing $s''\colon D^{p+q}_+\to S^m$ and $f\beta\colon D^{p+q}_-\to S^m$ together. Formally, put
$$
g(x)=
\begin{cases}
f\beta(x),               & \text{ for }x\in D^{p+q}_-; \\
s''(x)                  & \text{ for }x\in D^{p+q}_+.
\end{cases}
$$
The map $g$ is $C^1$-smooth because $f$ and $s$ coincide on $S^p\times S^q-\interior B$ and it is an embedding because $f(\interior B)$ is disjoint with $s(D^{p+1}\times S^q)\supset s''(D^{p+q}_+)$.

Equivalently, $g\sqcup s\colon S^{p+q}\sqcup S^p\times S^q\to S^m$ is obtained from $f\colon  S^p\times S^q\to S^m$ by an embedded surgery along an appropriate framed $(p+q)$-dimensional disc close to $s(B)$. This implies that $f$ is isotopic to an embedded connected sum of $g$ and $s$, because the connected summation is the ``inverse'' embedded surgery along a framed $1$-dimensional disc.

It remains to extend the embedding $g\colon S^{p+q}\to S^m$ to the link $g\colon S^{p+q}\sqcup S^q\to S^m$, whose restriction to 
$S^q$ is standard, 
and we get $\sigma^*(g)=f$.
\end{proof}

%


\subsection*{Exactness at $E^m(D^p\times S^q)$}


\begin{proof}[Proof that $\image i^*\subset \kernel \partial^*$]
Take an embedding $f\colon S^p\times S^q\to S^m$. By Lemma~\ref{lura}.b it follows that $f$ is isotopic to a $B$-standardized embedding
$f'\colon S^p\times S^q\to S^m$. By properties~(3') and~(4') of a $B$-standardized map it follows that the embedding $\partial^*i^*(f)=f'\beta\sqcup s\colon S^{p+q-1}\sqcup  S^{q-1}\to S^{m-1}$ extends to the embedding $f'\beta\sqcup s\colon D^{p+q-1}_-\sqcup D^{q}_-\to D^m_-$.
Thus the link $\partial^*i^*(f)$ is trivial by the analogue of Lemma~\ref{l2} for links. Hence $\partial^*i^*(f)=0$.
\end{proof}

\begin{proof}[Proof that $\image i^*\supset \kernel \partial^*$]
Let $f\colon D^p_-\times S^q\to S^m$ be an embedding such that $\partial^* (f)=0$. By Lemma~\ref{easy-extension}.a it extends to an embedding $F\colon S^p\times S^q-\interior B\to S^m$.
By Lemma~\ref{lura}.a the embedding $F$ is isotopic to a $B$-standardized embedding $F'\colon S^p\times S^q-\interior B\to S^m$. Since $\partial^*(f)=0$ it follows that the link $F'\beta\sqcup s\colon S^{p+q-1}\sqcup S^{q-1}$ is trivial. By Lemma~\ref{l-haefliger} the embedding $F'\beta\colon S^{p+q-1}\to \partial D^{m}_--s(D^{p+1}\times S^{q-1})$ extends to
a proper embedding $G\colon D^{p+q}_-\to D^m_- - s(D^{p+1}\times D^q_-)$ orthogonal to $S^{m-1}$. 
Gluing together the embeddings $F'\colon S^p\times S^q-\interior B\to S^m$ and $G\beta^{-1}\colon B\to D^m_-- s(D^{p+1}\times D^q_-)$ we get an embedding $G'\colon S^p\times S^q\to S^m$. Clearly, $i^*(G')=f$.
\end{proof}

\subsection*{Exactness at $E^m_0(S^{p+q}\sqcup S^q)$}

\begin{proof}[Proof that $\image \partial^*\subset \kernel \sigma^*$]
Let $f\colon D^{p}_-\times S^{q+1}\to S^{m+1}$ be an embedding. By Lemma~\ref{easy-extension}.a it extends to an embedding $F\colon S^p\times S^{q+1}-\interior B\to S^{m+1}$.
By Lemma~\ref{lura}.a the embedding $F$ is isotopic to a $B$-standardized embedding $F'\colon S^p\times S^{q+1}-\interior B\to S^{m+1}$.
Join $F'(\partial B)$ and $F'(D^p_+\times S^{q})$ by two general position arcs: $l\subset \partial D^{m+1}_+$ and $l'\subset F'(S^p\times D^{q+1}_+-\interior B)$.
Span the union $l\cup l'$ by a general position smooth $2$-dimensional disc $L\subset D^{m+1}_+$ orthogonal to $\partial D^{m+1}_+$ with corners at the endpoints of $l$ and $l'$. Take a framing of $L$ such that the first $p+q$ vectors at each point of the arc $l'$ are tangent to $F'(S^p\times D^{q+1}_+-\interior B)$.
Perform an embedded surgery over $F'\colon S^p\times D^{q+1}_+-\interior B\to D^{m+1}_+$ along the framed disc $L$. We get an embedding $G\colon S^p\times D^{q+1}\to D^{m+1}_+$ whose restriction to the boundary is
the connected sum of $F'\colon \partial B\to S^{m}$ and $F'\colon S^p\times S^{q}\to S^{m}$ along the arc $l$.
By Lemma~\ref{l2} it follows that the latter connected sum is isotopic to the standard embedding $S^p\times S^{q}\to S^{m}$. On the other hand, the connected sum is $\sigma^*\partial^*(f)$.
Thus $\sigma^*\partial^*(f)=0$.
\end{proof}

\begin{figure}[htp]
\includegraphics[width=17cm]{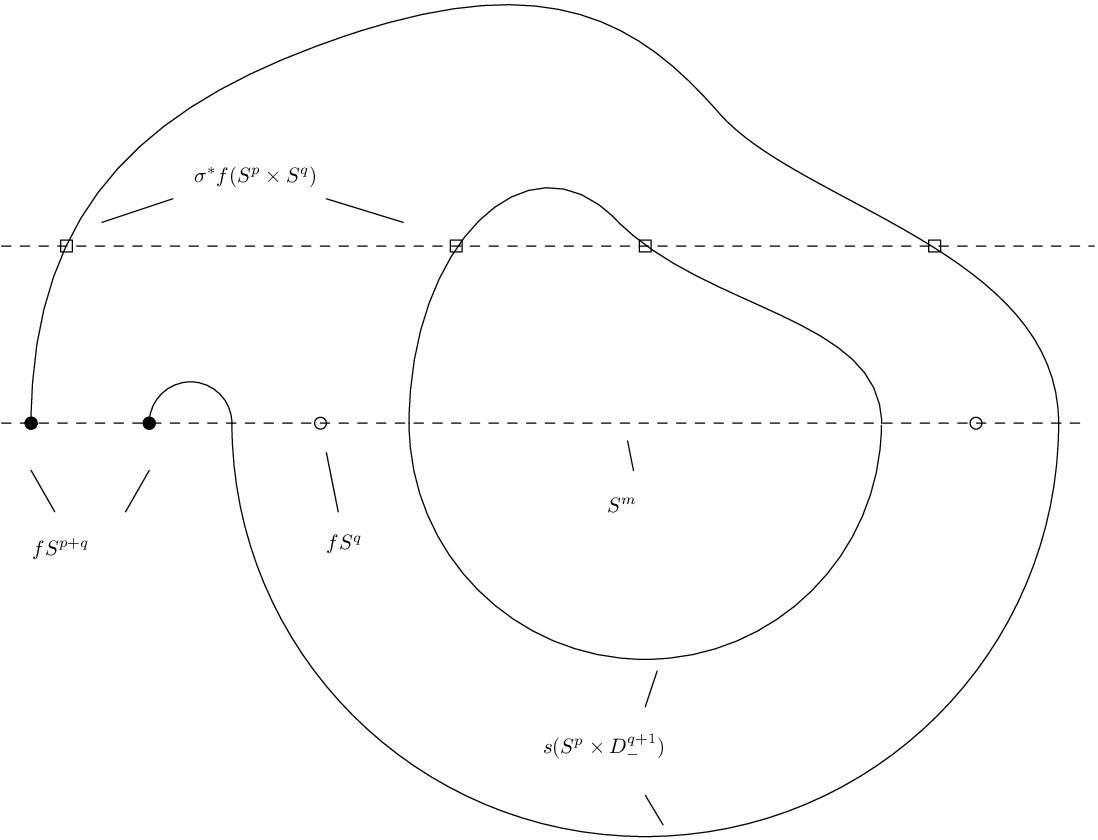}
\caption{The 
proof that $\image \partial^*\supset \kernel \sigma^*$ ($p=0, q=0, m=1$)}\label{exactness3}
\end{figure}

\begin{proof}[Proof that $\image \partial^*\supset \kernel \sigma^*$]
Take an embedding $f\colon S^{p+q}\sqcup S^{q}\to S^m$ whose restriction to $S^q$ is standard
and such that $fS^{p+q}\cap s(D^{p+1}\times S^q)=\emptyset$ and $\sigma^*(f)=0$.

Consider the embedding $\alpha_\epsilon\sigma^*(f)(\mathrm{Id}\times \alpha_\epsilon^{-1})\colon S^p\times \partial D^{q+1}_{+\epsilon}\to \partial D^{m+1}_{+\epsilon}$.
Since it is isotopic to the standard one by Lemma~\ref{l2} it follows that it extends to a proper embedding $F\colon S^p\times D^{q+1}_{+\epsilon}\to D^{m+1}_{+\epsilon}$. The embedding $F$ is shown in Figure~\ref{exactness3} above both dashed lines. Recall that
$\sigma^*(f)$ is a connected sum of $f\colon S^{p+q}\to S^m$ and $s\colon S^{p}\times S^{q}\to S^m$ along an arc.
Attach the trace of the surgery performing this connected summation to the embedding~$F$. This trace is shown in Figure~\ref{exactness3} between the dashed lines. We get a proper embedding
$F'\colon S^p\times D^{q+1}_+ - \interior B\to D^{m+1}_+$, whose restriction to the boundary is the union
of $f\colon S^{p+q}\to S^{m}$ and $s\colon S^p\times S^q\to S^m$.
Attach the standard embedding $s\colon S^{p}\times D^{q+1}_-\to D^{m+1}_-$
to the embedding $F'$. This embedding is shown in Figure~\ref{exactness3} below both dashed lines. We get an embedding $F''\colon S^p\times S^{q+1}-\interior B\to S^{m+1}$. 

It follows by definition that the latter embedding is $B$-standardized.

Define the embedding $g\colon D^p_-\times S^{q+1}\to S^{m+1}$ to be the restriction of $F''\colon S^p\times S^{q+1}-\interior B\to S^{m+1}$.
Then by definition $\partial^*(g)=f$.
\end{proof}

The proof of Theorem~\ref{th3} is completed.

\section{Applications}\label{sect3}





\subsection*{Finiteness criterion}

Now let us apply the sequence of Theorem~\ref{th3} to determine precisely when the set $E^m(S^p\times S^q)$ is finite, i.e., to prove Theorem~\ref{th2}.

%
%
%



Denote by $E^m_{\mathrm{U}}(S^p\sqcup S^q)$ be the group of isotopy classes of smooth embeddings $S^p\sqcup S^q\to S^m$, whose restrictions to \emph{both} components $S^p$ and $S^q$ are unknotted. For a finitely generated Abelian group $G$ identify $G\otimes \mathbb{Q}$ with $\mathbb{Q}^{\rank G}$.
We are going to use tacitly the following isomorphisms (see \cite[Theorem 2.4]{Hae66C} and \cite[Theorem~1.13]{CFS11} respectively):
\begin{align*}
&E^{m}(S^{p}\sqcup S^{q})
\cong
E^{m}_0(S^{p}\sqcup S^{q})\oplus E^{m}(S^{q})
\cong
E^m_{\mathrm{U}}(S^{p}\sqcup S^q)\oplus E^{m}(S^{p})\oplus E^{m}(S^{q}),\\
&E^m(D^p\times S^q)\otimes\mathbb{Q} \cong E^m(S^q)\otimes\mathbb{Q}\oplus \pi_q(V_{m-q,p})\otimes\mathbb{Q}
\end{align*}
holding for $p,q< m-2$ and $p+q < m-2$, respectively.

\begin{proof}[Proof of Corollary~\ref{essentials}]
Let us prove that the map $\partial^*\colon  E^m(D^p\times S^q)\to E^{m-1}_0(S^{p+q-1}\sqcup S^{q-1})$ has finite image for $m>2p+q+2$, $m< p+{3q}/{2}+2$.
If $p=0$ then this follows immediately from Theorem~\ref{th3} because the map $i^*\colon E^m(S^0\times S^q)\to E^m(D^0\times S^q)$ is surjective.
Assume further that $p\ge1$. By the first definition of $\partial^*$ it suffices to prove that at least one of the groups $E^m(D^p\times S^q)$ and $\pi_{q-1}(S^{m-p-q-1})$ is finite. The assumptions $m>2p+q+2$ and $m<p+{3q}/{2}+2$ imply that $m\le 2q$. So by Theorem~\ref{th1f} the group
$E^m(D^p\times S^q)$ is finite unless $4\,|\,q+1$. By the Serre finiteness criterion for homotopy groups of spheres the group $\pi_{q-1}(S^{m-p-q-1})$ is finite unless $4\,|\,q$. So the map $\partial^*\colon E^m(D^p\times S^q)\to E^{m-1}_0(S^{p+q-1}\sqcup S^{q-1})$ has finite image. Analogously the map $\partial^*\colon E^{m+1}(D^p\times S^{q+1})\to E^{m}_0(S^{p+q}\sqcup S^{q})$ has finite image because the inequality  $m\le 2q$ implies  $m+1\le 2(q+1)$.

This implies that the sequence of Theorem~\ref{th3} tensored by $\mathbb{Q}$ splits for $m<p+3q/{2}+2$. Thus $E^m(S^p\times S^q)\otimes\mathbb{Q}\cong E^{m}_0(S^{p+q}\sqcup S^{q})\otimes\mathbb{Q}\oplus E^m(D^p\times S^q)\otimes\mathbb{Q}$. By the isomorphisms stated in the beginning of this section
the corollary follows.
\end{proof}


\begin{proof}[Proof of Theorem~\ref{th2}]
(1) {\it ``Only if'' part.} 
If neither conditions in the list of Theorem~\ref{th2} hold then 
$E^m_{\mathrm{U}}(S^{p+q}\sqcup S^q)$, $E^{m}(S^{p+q})$,
and $E^m(D^{p}\times S^{q})$
are finite by Theorems~\ref{th1}--\ref{th1b}.
By Corollary~\ref{essentials} (or alternatively by the exactness at the term~$E^m(S^p\times S^q)$ in Theorem~\ref{th3}) the result follows.

(2) {\it Case when $q+1$ is divisible by $4$}. 
If $m\le p+{3}q/{2}+{1}/{2}$ then by Theorem~\ref{th1f} the group $E^{m}(D^p\times S^q)$ is infinite.
If $m= p+{3}q/{2}+{3}/{2}$
then $2(m-p-q-2)+(m-q-2)=(m-3)$ and $m-p-q$ is odd; thus $(2;1)\in FCS (m-p-q,m-q)$ and by Theorem~\ref{th1b} the group $E^m_{\mathrm{U}}(S^{p+q}\sqcup S^q)\subset E^m_{0}(S^{p+q}\sqcup S^q)$ is infinite. By Corollary~\ref{essentials}
the result follows.

(3) {\it The rest of ``if'' part}. If $p+q+1$ is divisible by $4$ then by Theorem~\ref{th1} the group $E^{m}(S^{p+q})$ is infinite because $m<p+{3q}/{2}+2$.
If there is a point $(x,y)\in FCS (m-p-q,m-q)$ such that $(m-p-q)x+(m-q)y=(m-3)$ then by Theorem~\ref{th1b} the group $E^m_{\mathrm{U}}(S^{p+q}\sqcup S^q)$ is infinite. By Corollary~\ref{essentials} and the isomorphisms stated at the beginning of this section
the result follows.
\end{proof}

\subsection*{Knotted connected sums}

Let us give the following easy example of another application of the standard surgery; see~\S\ref{sectprem}.

\begin{theorem}\label{thconsum} For each $q_1\ge p_1\ge p_2$, $m > 2p_1+q_1+2$, $q_2+p_2=q_1+p_1$ there is an exact sequence 
\begin{equation*}
{{E}^{m}(S^{p_1+q_1})} \xrightarrow{\sigma^*}
{{E}^m(S^{p_1}\times S^{q_1})} \oplus {{E}^{m}(S^{p_2}\times S^{q_2}}) \xrightarrow{a}
{{E}^m(S^{p_1}\times S^{q_1} \,\# \,S^{p_2}\times S^{q_2})}.
\end{equation*}
\end{theorem}


Here ${{E}^m(S^{p_1}\times S^{q_1} \,\# \,S^{p_2}\times S^{q_2})}$ is a set with the marked element --- the connected sum of two embeddings $S^{p_i}\times S^{q_i}\to S^m$ isotopic to the standard ones and separated by a hyperplane.
This result and Lemma~\ref{cl1} above imply that if at least one set ${{E}^m(S^{p_i}\times S^{q_i})}$ is infinite then the set
${E}^m(S^{p_1}\times S^{q_1}\, \# \, S^{p_2}\times S^{q_2})$ is infinite. 



For the proof of Theorem~\ref{thconsum} we need a definition and a lemma.
Let $N$ be a closed connected $n$-manifold. Denote by $c\colon  S^p\times D^{n-p}_- \to N$ a fixed embedding.
A map $f \colon  N \to S^m$ is called \emph{$c$-standardized}, if
\begin{itemize}
\item $fc \colon  S^p \times D^{n-p}_- \to D^m_-$ is standard; 
\item $f(N-c(S^p\times D^{n-p}_-))\subset \interior D^m_+$.
\end{itemize}
An isotopy $f_t \colon  N\to S^m$ is {\it $c$-standardized},
if for each $t\in I$ the embedding $f_t$ is $c$-standardized.

\begin{lemma}\label{lstand} \textup{\cite[Standardization Lemma]{Sko07F}} Assume that $m> n+p+2$. Then
any embedding
$N\to S^m$ is isotopic to a $c$-standardized embedding.
%
\end{lemma}

This lemma allows to define an action $\#_c\colon E^m(S^p\times S^q)\times E^m(N)\to E^m(N)$
analogously to the group structure on $E^m(S^p\times S^q)$ in \S\ref{sectprem}, see \cite{Sko07F} for details. However, we do not need this action for our proof.


\begin{proof}[Proof of Theorem~\ref{thconsum}]
Fix two embeddings  $s_i\colon S^{p_i}\times S^{q_i}\to S^m$  isotopic to the standard ones and separated by a hyperplane. Let $c_i\colon  S^{p_i}\times D^{q_i}_-\to S^{p_1}\times S^{q_1} \,\# \,S^{p_2}\times S^{q_2}$ be the obvious inclusions.

The map $a\colon {E}^m(S^{p_1}\times S^{q_1}) \oplus {E}^{m}(S^{p_2}\times S^{q_2})\to {E}^m(S^{p_1}\times S^{q_1}\# S^{p_2}\times S^{q_2})$
is the embedded connected summation.
The homomorphism $\sigma^*\colon E^m(S^{p_1+q_1})\to {E}^m(S^{p_1}\times S^{q_1}) \oplus {E}^{m}(S^{p_2}\times S^{q_2})$
is defined by the formula $\sigma^*(f)=(s_1\#f,s_2\#(-f))$, where $-f=r_m fr_{p_1+q_1}$ is the inverse in the group $E^m(S^{p_1+q_1})$.

To prove the exactness we need to show that $a\sigma^*$ is a constant map to the marked element and $\sigma^*$ surjects onto the preimage of the marked element.

For any $f\in E^m(S^{p+q})$ the embedding $a\sigma^*(f)=s_1\#f\#s_2\#(-f)=s_1\#s_2$ is the marked element. 

Let us prove that $\sigma^*$ surjects onto the preimage of $s_1\#s_2$. Let $\bar\sigma^*_i\colon {E}^m(S^{p_i}\times S^{q_i})\to E^m(S^{p+q})$, where $i=1,2$,
be the map defined in the proof of Lemma~\ref{cl1}.

Let us define also a map $\bar a_1\colon {{E}^m(S^{p_1}\times S^{q_1} \,\# \,S^{p_2}\times S^{q_2})} \to {{E}^m(S^{p_1}\times S^{q_1})}$.  Take an embedding $f\colon S^{p_1}\times S^{q_1} \,\# \,S^{p_2}\times S^{q_2}\to S^m$.
By Lemma~\ref{lstand} it is isotopic to a $c_2$-standardized embedding $f'\colon S^{p_1}\times S^{q_1} \,\# \,S^{p_2}\times S^{q_2}\to S^m$.
Perform the standard embedded surgery over $f'$ along the meridian $f'c_2(S^{p_2}\times y)$, $y\in S^{q_2}$; see \S2. Set $\bar a_1(f)\colon S^{p_1}\times S^{q_1}\to S^m$  to be the obtained  embedding.

Take any pair of embeddings $f_i\colon S^{p_i}\times S^{q_i}\to S^m$, where $i=1,2$. It is easy to see that
$\bar a_1(f_1\# f_2)=f_1\# \bar\sigma_2^* f_2$ and $\bar\sigma_1^*\bar a_1(f_1\#f_2)=\bar\sigma_1^* f_1\# \bar\sigma_2^* f_2$. Now assume that $a(f_1,f_2)=f_1\#f_2=s_1\#s_2$ is the marked element. Then $\bar a_1(f_1\# f_2)=s_1$ and $\bar\sigma_1^*\bar a_1(f_1\#f_2)=0$. Thus $f_1\# \bar\sigma_2^* f_2=s_1$ and
$\bar\sigma_1^* f_1\# \bar\sigma_2^* f_2=0$. Hence $f_1=s_1\#(-\bar\sigma_2^* f_2)=s_1\#(\bar\sigma_1^*f_1)$.
Analogously $f_2=s_2\#(-\bar\sigma_1^* f_1)$. So $(f_1,f_2)=\sigma^*(\bar\sigma_1^*f_1)$ belongs to the image of $\sigma^*$, which proves the theorem.
\end{proof}

\subsection*{Relation to another exact sequence}

\begin{remark}\label{rem-desuspension}
The exact sequence of Theorem~\ref{th3} above is mapped to the middle horizontal sequence of \cite[Restriction Lemma 5.2]{Sko08} as follows. Use the notation from \cite{Sko08}. First, the two sequences in question have a common term $E^m(D^p \times S^q)=KT^m_{p,q,+}$. Further, there is an obvious 'forgetful' map $e\colon E^m(S^p \times S^q)\to \overline{KT}^m_{p,q}$; see \cite[Section~2.5 and Theorem~3]{CRS08} (in the latter paper the notation $\bar E^m(S^p \times S^q)=\overline{ KT}^m_{p,q}$ is used). Finally, the 'linking number' map $\lambda\colon E^m_0(S^{p+q}\sqcup S^q)\to\pi_{p+q}(S^{m-q-1})$ takes a link to the homotopy class of the $(p+q)$-dimensional component in the complement to the $q$-dimensional one.

The maps $e$ and $\lambda$ together with the two exact sequences themselves form a diagram, the commutativity of which follows directly by definitions. The maps $e$ and $\lambda$ are isomorphisms for $m\ge 3(p+q)/2+2$ by the classification results of \cite{Sko02} and \cite{Hae66C} respectively. Thus the two exact sequences are isomorphic for $m\ge 3(p+q)/2+2$.

In general one sequence is mapped to the other. The maps $e$ and $\lambda$ have many instrumental properties analogous to those of the suspension map, e.g., the existence of exact EHP sequences \cite[Lemma 2]{CRS08}, \cite[Theorem 3.1]{Sko08P} (the map $e$ in the latter theorem coincides with $\lambda$ up to isomorphism). So the exact sequence of Theorem~\ref{th3} above can be informally considered as a 'desuspension' of the middle horizontal sequence of \cite[Restriction Lemma 5.2]{Sko08}.

The main new ingredients in the proof of Theorem~\ref{th3} in comparison to \cite[Restriction Lemma~5.2]{Sko08} are the constructions of the maps $\sigma^*$ and $\partial^*$ (see Figures~\ref{fig6} and~\ref{altdelta} respectively) and removing the intersection $f(B)\cap s(\interior D^{p+1}\times S^q)$ (see Figure~\ref{exactness1}). 
\end{remark}

\subsection*{Underwater reefs}


Let us make a few corrections to the previous paper \cite{CRS08} on the subject.

The definition of the standard embedding $S^p\times D^q_+\to D^m_+$ in \cite[Definition~6]{CRS08} was incorrect because it did not give a proper embedding. A correct definition is given above in \S2.

The construction of the group structure on the set of knotted tori in \cite[\S5]{CRS08} was incomplete because the analogues of Lemmas~\ref{l1}--\ref{l2} were not proved.
A complete construction is given in~\cite{Sko15}.


Step 1) of the proof of \textup{\cite[Lemma~1]{CRS08}} used the following result without proof, which we give now.

\begin{lemma} Assume that $m\ge p+\frac{4}{3}q+2$ and $f\colon S^p\times D^q_+\to D^m_+$ is a proper general position smooth map (possibly, with self-intersections). Then  the suspension map
$$
\Sigma^\infty\colon\pi_q(D^m_+-f(S^p\times D^q_+),\partial D^m_+-f(S^p\times \partial D^q_+))\to \{S^q,S^m-f(S^p\times D^q_+)\}
$$
is bijective.
\end{lemma}


\begin{proof}
By the Alexander duality the complement $\partial D^m_+-f(S^p\times \partial D^q_+)$ is $(m-p-q-1)$-connected. By \cite[Proposition~5]{CRS08} the pair $(D^m_+-f(S^p\times D^q_+),\partial D^m_+-f(S^p\times \partial D^q_+)$ is $c$-connected, where $c=\min\{m-p-2,2m-2p-2q-3\}$. The assumption $m\ge p+\frac{4}{3}q+2$ implies that $q<c+m-p-q-1$ and $q/2\le c$.
Then by homotopy excision theorem for each $r\le q$ the map $\pi_r(D^m_+-f(S^p\times D^q_+),\partial D^m_+-f(S^p\times \partial D^q_+))\to \pi_r(S^m-f(S^p\times \partial D^q_+))$
is bijective. In particular, $S^m-f(S^p\times D^q_+)$ is $q/2$-connected. Then by suspension theorem $\pi_q(S^m-f(S^p\times \partial D^q_+))\to \{S^q,S^m-f(S^p\times D^q_+)\}$
is bijective, and the lemma follows.
\end{proof}

\subsection*{Open problems}
There are many similar open questions (by A.~Skopenkov):
\begin{enumerate}
\item Does the set of {\it piecewise linear} embeddings $S^p\times S^q\to S^m$ up to piecewise linear isotopy admit a natural group structure for each $m> p+q+2$? Can the restriction $m> 2p+q+2$ be weakened to $m> p+q+2$ in Theorem~\ref{th3} in the piecewise linear category?

\item How many embeddings $S^1\times S^5\to S^{10}$ are there up to isotopy? (This problem have been recently solved  \cite[end of \S2]{Sko15}.) Find more explicit classification results.

\item Is it true that for $p\ge 1$ and $m>2p+q+2$ there is an isomorphism
    $$
    E^m(S^p\times S^q)\cong
    E^m(D^{p+1}\times S^q)\oplus E^m(S^{p+q})\oplus\kernel\lambda,
    $$
    where $\lambda\colon E^m_{\mathrm{U}}(S^{p+q}\sqcup S^q)\to\pi_q(S^{m-p-q-1})$ is the linking number?
\item When is the set of embeddings finite? Find more finiteness results for the sets $E^m(N)$.
\end{enumerate}

\subsection*{Acknowledgements}

The author is grateful to A. Skopenkov for continuous attention to the work and also to P. Akhmetiev,  D. Crowley, G. Laures, U. Kaiser, U. Koschorke, S. Melikhov, A. Mischenko, V.~Nezhinsky, and A.~Zhubr for useful discussions. The author is grateful to King Abdullah University of Science and Technology for hosting him during one of the periods of the work over the paper.

\bibliographystyle{amsplain}


\noindent
\textsc{Mikhail Skopenkov\\
National Research University Higher School of Economics\\
and\\
Institute for information transmission problems\\
of the Russian Academy of Sciences \\
}
\texttt{skopenkov@rambler.ru} \quad \url{http://skopenkov.ru}

\end{document}